\documentclass[a4paper,reqno,twoside]{amsart}

\usepackage{a4wide}
\usepackage{amssymb}
\usepackage{latexsym}
\usepackage{amsmath}

\usepackage{geometry}
\usepackage[UKenglish]{babel}
\usepackage{color}
\usepackage[bookmarks=true,
            bookmarksnumbered=true,
            colorlinks=false,
            pdfstartview=FitH]{hyperref}
 \hypersetup{pdftitle={Additive units of product systems},
 pdfauthor={Bhat, Lindsay and Mukherjee}}
\usepackage{comment} %added

\newtheorem{theorem}{Theorem}[section]
\newtheorem{lemma}[theorem]{Lemma}
\newtheorem{proposition}[theorem]{Proposition}
\newtheorem{corollary}[theorem]{Corollary}
\theoremstyle{definition}
\newtheorem{definition}[theorem]{Definition}
\newtheorem{example}[theorem]{Example}

\newtheorem*{remark}{Remark}
\newtheorem*{remarks}{Remarks}
\newtheorem*{notation}{Notation}
\numberwithin{equation}{section}

\newcommand{\wedgewithlimits}{\operatornamewithlimits{\wedge}}

\newsavebox\dotbox
\sbox{\dotbox}{\(\displaystyle\bigodot\)}
\newlength{\dotheight}
\DeclareMathOperator*{\bigcdot}{\boldsymbol{\cdot}\rule[-\dp\dotbox]{0pt}{\dotheight}}
\font\germ=eufm10
\newcommand{\Hil}{\mathsf{H}}
\newcommand{\hil}{\mathsf{h}}
\newcommand{\Kil}{\mathsf{K}}
\newcommand{\kil}{\mathsf{k}}

\newcommand{\Fock}{\mathcal{F}}

\newcommand{\Focks}{\mathcal{F}_s}
\newcommand{\Focksminusr}{\mathcal{F}_{s-r}}
\newcommand{\Fockt}{\mathcal{F}_t}
\newcommand{\Fockrminuss}{\mathcal{F}_{r-s}}
\newcommand{\Focksplustminusr}{\mathcal{F}_{s+t-r}}
\newcommand{\Focksplust}{\mathcal{F}_{s+t}}

\newcommand{\Fockk}{\mathcal{F}^\kil}
\newcommand{\Fockh}{\mathcal{F}^\hil}
\newcommand{\Fockkone}{\mathcal{F}^{\kil_1}}
\newcommand{\Fockktwo}{\mathcal{F}^{\kil_2}}
\newcommand{\Fockkt}{\mathcal{F}^\kil_t}
\newcommand{\Fockkinfty}{\mathcal{F}^\kil_\infty}
\newcommand{\Guich}{\mathcal{G}}
\newcommand{\Guichk}{\mathcal{G}^\kil}
\newcommand{\Guichkt}{\mathcal{G}^\kil_t}
\newcommand{\Guichksplust}{\mathcal{G}^\kil_{s+t}}
\newcommand{\Guichkinfty}{\mathcal{G}^\kil_\infty}

\newcommand{\Real}{\mathbb{R}}
\newcommand{\Rplus}{\Real_+}
\newcommand{\Comp}{\mathbb{C}}
\newcommand{\Nat}{\mathbb{N}}

\newcommand{\Rat}{\mathbb{Q}}

\newcommand{\ve}{\varepsilon}
\newcommand{\vp}{\varpi}

\newcommand{\ip}[2]{\langle #1, #2 \rangle}

\newcommand{\norm}[1]{\lVert #1 \rVert}

\newcommand{\dyad}[2]{| #1 \rangle \langle #2 |}

\newcommand{\Vac}{\Omega}
\newcommand{\Vack}{\Omega^\kil}
\newcommand{\Vackt}{\Omega^\kil_t}

\newcommand{\wh}{\widehat}
\newcommand{\wt}{\widetilde}
\newcommand{\ol}{\overline}
\newcommand{\ot}{\otimes}

\newcommand{\op}{\oplus}

\newcommand{\les}{\leqslant}
\newcommand{\ges}{\geqslant}

\newcommand{\ti}{\textit}
\newcommand{\tu}{\textup}

\newcommand{\bfcdot}{{\bf \bigcdot}}
\newcommand{\CP}{C_{\mathcal{P}}}
\newcommand{\Prt}{P_{r,t}}
\newcommand{\Pst}{P_{s,t}}
\newcommand{\Prs}{P_{r,s}}
\newcommand{\Psiti}{P_{s_i,t_i}}

\newcommand{\Mst}{M_{[s,t]}}

\newcommand{\Msiti}{M_{[s_i,t_i]}}

\newcommand{\PF}{P^\Fock}
\newcommand{\PFab}{P^\Fock_{a,b}}
\newcommand{\PFrt}{P^\Fock_{r,t}}
\newcommand{\PFst}{P^\Fock_{s,t}}

\newcommand{\PFut}{P^\Fock_{u,t}}

\newcommand{\PFsu}{P^\Fock_{s,u}}

\newcommand{\PFcheckst}{P^{\Fock\,\check{}}_{s,t}}

\newcommand{\PFchecksiti}{P^{\Fock\,\check{}}_{s_i,t_i}}
\newcommand{\Prob}{\mathbb{P}}
\newcommand{\ProbPomega}{\mathbb{P}^P_\omega}
\newcommand{\ProbPphi}{\mathbb{P}^P_{\varphi}}
\newcommand{\ProbFomega}{\mathbb{P}^\Fock_\omega}
\newcommand{\ProbFcheckomega}{\mathbb{P}^{\Fock\,\check{}}_\omega}
\newcommand{\closed}{\mathcal{C}}
\newcommand{\piP}{\pi^P}
\newcommand{\piF}{\pi^\Fock}
\newcommand{\MF}{\mathcal{M}^\Fock}
\newcommand{\MFcheck}{\mathcal{M}^{\Fock\,\check{}}}

\newcommand{\Nullset}{\mathcal{N}}

\newcommand{\BFkst}{B^{\Fock,\kil}_{s,t}}

\newcommand{\BGkst}{B^{\Guich,\kil}_{s,t}}

\newcommand{\uhat}{\wh{u}}
\newcommand{\vhat}{\wh{v}}
\newcommand{\Ahat}{\wh{A}}
\newcommand{\ahat}{\wh{a}}

\newcommand{\E}{\mathcal{E}}
\newcommand{\Ei}{\mathcal{E}^i}
\newcommand{\ET}{\mathcal{E}_T}
\newcommand{\ER}{\mathcal{E}_R}
\newcommand{\ES}{\mathcal{E}_S}
\newcommand{\ERplusS}{\mathcal{E}_{R+S}}
\newcommand{\ERplusSplusT}{\mathcal{E}_{R+S+T}}
\newcommand{\ETminust}{\mathcal{E}_{T-t}}
\newcommand{\Etminusr}{\mathcal{E}_{t-r}}
\newcommand{\Er}{\mathcal{E}_r}
\newcommand{\Es}{\mathcal{E}_s}
\newcommand{\Et}{\mathcal{E}_t}

\newcommand{\Esminusr}{\mathcal{E}_{s-r}}
\newcommand{\Erminuss}{\mathcal{E}_{r-s}}
\newcommand{\Esplust}{\mathcal{E}_{s+t}}
\newcommand{\BE}{B^\mathcal{E}}
\newcommand{\BEst}{B^\mathcal{E}_{s,t}}
\newcommand{\BErs}{B^\mathcal{E}_{r,s}}
\newcommand{\BErsplust}{B^\mathcal{E}_{r,s+t}}
\newcommand{\BErplusst}{B^\mathcal{E}_{r+s,t}}
\newcommand{\BEST}{B^\mathcal{E}_{S,T}}
\newcommand{\BERS}{B^\mathcal{E}_{R,S}}
\newcommand{\BERSplusT}{B^\mathcal{E}_{R,S+T}}
\newcommand{\BERplusST}{B^\mathcal{E}_{R+S,T}}
\newcommand{\BEtTminust}{B^\mathcal{E}_{t,T-t}}
\newcommand{\BETminustt}{B^\mathcal{E}_{T-t,t}}
\newcommand{\BEonest}{B^{\Eone}_{s,t}}
\newcommand{\BFST}{B^\mathcal{F}_{S,T}}

\newcommand{\Bst}{B_{s,t}}
\newcommand{\BST}{B_{S,T}}

\newcommand{\betaE}{\beta^E}
\newcommand{\betaEst}{\beta^E_{s,t}}
\newcommand{\betaErs}{\beta^E_{r,s}}
\newcommand{\betaErsplust}{\beta^E_{r,s+t}}
\newcommand{\betaErplusst}{\beta^E_{r+s,t}}

\newcommand{\betaEbfst}{\beta^E_{{\bf s,t}}}
\newcommand{\betaEbfrs}{\beta^E_{{\bf r,s}}}
\newcommand{\betaFbfrs}{\beta^F_{{\bf r,s}}}
\newcommand{\betaEbfrt}{\beta^E_{{\bf r,t}}}
\newcommand{\betaEbfss}{\beta^E_{{\bf s,s}}}
\newcommand{\betaEbftt}{\beta^E_{{\bf t,t}}}
\newcommand{\betaEbfuv}{\beta^E_{{\bf u,v}}}

\newcommand{\betast}{\beta_{s,t}}

\newcommand{\betabfst}{\beta_{{\bf s,t}}}

\newcommand{\Ebfs}{E_{{\bf s}}}
\newcommand{\Ebft}{E_{{\bf t}}}
\newcommand{\Fbfr}{F_{{\bf r}}}
\newcommand{\Fbfs}{F_{{\bf s}}}
\newcommand{\Fbft}{F_{{\bf t}}}

\newcommand{\FSplusT}{\mathcal{F}_{S+T}}

\newcommand{\betaERbfr}{\beta^E_{R, {\bf r}}}
\newcommand{\BETbfs}{B^{\E}_{T, {\bf s}}}
\newcommand{\BETbfr}{B^{\E}_{T, {\bf r}}}
\newcommand{\BESbfs}{B^{\E}_{S, {\bf s}}}
\newcommand{\BETbft}{B^{\E}_{T, {\bf t}}}

\newcommand{\BTbft}{B_{T, {\bf t}}}
\newcommand{\betaTbft}{\beta_{T, {\bf t}}}

\newcommand{\BEbfrs}{B^{\mathcal{E}}_{{\bf r,s}}}

\newcommand{\iE}{\imath^E}
\newcommand{\iEbft}{\imath^E_{{\bf t}}}
\newcommand{\iEbfs}{\imath^E_{{\bf s}}}
\newcommand{\iEbfr}{\imath^E_{{\bf r}}}
\newcommand{\iEt}{\imath^E_{t}}
\newcommand{\iFt}{\imath^F_{t}}

\newcommand{\ibft}{\imath_{{\bf t}}}
\newcommand{\ibfs}{\imath_{{\bf s}}}
\newcommand{\ibfr}{\imath_{{\bf r}}}

\newcommand{\abft}{a_{{\bf t}}}
\newcommand{\abfs}{a_{{\bf s}}}
\newcommand{\abfr}{a_{{\bf r}}}

\newcommand{\ubft}{u_{{\bf t}}}
\newcommand{\ubfs}{u_{{\bf s}}}
\newcommand{\ubfr}{u_{{\bf r}}}

\newcommand{\imt}{\imath_{t}}
\newcommand{\is}{\imath_{s}}

\newcommand{\isplust}{\imath_{s+t}}

\newcommand{\iFbft}{\imath^F_{{\bf t}}}
\newcommand{\iFbfs}{\imath^F_{{\bf s}}}

\newcommand{\betaFst}{\beta^F_{s,t}}

\newcommand{\betaFkst}{\beta^{F,\kil}_{s,t}}
\newcommand{\betaFkrs}{\beta^{F,\kil}_{r,s}}
\newcommand{\betaFkrsplust}{\beta^{F,\kil}_{r,s+t}}
\newcommand{\betaFkrplusst}{\beta^{F,\kil}_{r+s,t}}
\newcommand{\WEst}{W^\mathcal{E}_{s,t}}

\newcommand{\IEone}{I^\mathcal{E}_{1}}
\newcommand{\IEa}{I^\mathcal{E}_{a}}
\newcommand{\IEr}{I^\mathcal{E}_{r}}
\newcommand{\IEs}{I^\mathcal{E}_{s}}
\newcommand{\IEt}{I^\mathcal{E}_{t}}

\newcommand{\IEoneminust}{I^\mathcal{E}_{1-t}}
\newcommand{\IEoneminusa}{I^\mathcal{E}_{1-a}}
\newcommand{\IER}{I^\mathcal{E}_{R}}
\newcommand{\IET}{I^\mathcal{E}_{T}}
\newcommand{\UET}{U^{\mathcal{E},T}}
 \newcommand{\UETt}{U^{\mathcal{E},T}_{t}}
 \newcommand{\UnitE}{\mathcal{U}^{\mathcal{E}}}
  \newcommand{\UnitEtensorF}{\mathcal{U}^{\mathcal{E} \ot \F}}
  \newcommand{\UnitEnorm}{\mathcal{U}^{\mathcal{E}}_1}
  \newcommand{\UnitEtwonorm}{\mathcal{U}^{\Etwo}_1}
   \newcommand{\AdditEu}{A^{\mathcal{E}}_u}
      \newcommand{\RootEu}{R^{\mathcal{E}}_u}
            \newcommand{\RootEtensorFutensorv}{R^{\E \ot \F}_{u \ot v}}
         \newcommand{\RootFv}{R^{\mathcal{F}}_v}
       \newcommand{\RootEoneuone}{R_{u^1}^{\Eone}}
              
             \newcommand{\RootEtwoutwo}{R^{\Etwo}_{u^2}}
               
         \newcommand{\TrivEu}{T^{\mathcal{E}}_u}
      \newcommand{\RootEucal}{\mathcal{R}^{\mathcal{E}}_u}
            \newcommand{\RootEoneuonecal}{\mathcal{R}_{u^1}^{\Eone}}
             \newcommand{\RootEtwoutwocal}{\mathcal{R}^{\Etwo}_{u^2}}
             \newcommand{\RootEoneucal}{\mathcal{R}_{u}^{\Eone}}
             \newcommand{\RootEtwoucal}{\mathcal{R}^{\Etwo}_{u}}
 \newcommand{\AdditEv}{A^{\mathcal{E}}_v}
      
         \newcommand{\RootFkVac}{R^{\F, \kil}_\Vac}
         
           \newcommand{\RootOmega}{R_\Omega}
         
  \newcommand{\EI}{\mathcal{E}^I}
  
    \newcommand{\EIT}{\mathcal{E}^I_T}
 \newcommand{\F}{\mathcal{F}}
  \newcommand{\G}{\mathcal{G}}
    \newcommand{\Gt}{\mathcal{G}_t}
      \newcommand{\Gone}{\mathcal{G}_1}
        \newcommand{\Goneovern}{\mathcal{G}_{1/n}}
         \newcommand{\Foneovern}{\mathcal{F}_{1/n}}
         \newcommand{\Eoneovern}{\mathcal{E}_{1/n}}
\newcommand{\FT}{\mathcal{F}_T}

\newcommand{\Fs}{\mathcal{F}_s}
\newcommand{\Ft}{\mathcal{F}_t}

\newcommand{\Fsplust}{\mathcal{F}_{s+t}}
\newcommand{\BFst}{B^\mathcal{F}_{s,t}}

  \newcommand{\UnitF}{\mathcal{U}^{\mathcal{F}}}
  \newcommand{\UnitFnorm}{\mathcal{U}^{\mathcal{F}}_1}
    \newcommand{\FI}{\mathcal{F}^I}

\newcommand{\Eone}{\mathcal{E}^1}
\newcommand{\Eonet}{\mathcal{E}^1_t}

\newcommand{\Etwo}{\mathcal{E}^2}
\newcommand{\Etwot}{\mathcal{E}^2_t}

\newcommand{\thetaEut}{\theta^{\mathcal{E},u}_t}

\newcommand{\Fminusperp}{F^{\ominus \perp}}
\newcommand{\Fminusperptminuss}{F^{\ominus \perp}_{t-s}}

\newcommand{\Fminuss}{F^{\ominus}_s}
\newcommand{\Fminust}{F^{\ominus}_t}
\newcommand{\Fminussplust}{F^{\ominus}_{s+t}}

\newcommand{\Fminustminuss}{F^{\ominus}_{t-s}}
\newcommand{\Fminusperps}{F^{\ominus \perp}_s}
\newcommand{\Fminusperpt}{F^{\ominus \perp}_t}
\newcommand{\Fminusperpsplust}{F^{\ominus \perp}_{s+t}}

\newcommand{\FFminusperp}{\mathcal{F}^{\ominus \perp}}

\newcommand{\FFminuss}{\mathcal{F}^{\ominus}_s}

\newcommand{\FFminussplust}{\mathcal{F}^{\ominus}_{s+t}}

\newcommand{\FFminusperps}{\mathcal{F}^{\ominus \perp}_s}
\newcommand{\FFminusperpt}{\mathcal{F}^{\ominus \perp}_t}
\newcommand{\FFminusperpsplust}{\mathcal{F}^{\ominus \perp}_{s+t}}

\newcommand{\FFperpr}{\mathcal{F}^{\perp}_r}
\newcommand{\FFperps}{\mathcal{F}^{\perp}_s}
\newcommand{\FFperpt}{\mathcal{F}^{\perp}_t}
\newcommand{\FFperpsplustminusr}{\mathcal{F}^{\perp}_{s+t-r}}
\newcommand{\FFperpsminusr}{\mathcal{F}^{\perp}_{s-r}}
\newcommand{\FFperptminusr}{\mathcal{F}^{\perp}_{t-r}}

\newcommand{\XEut}{X^{\E,u}_t}
\newcommand{\XEus}{X^{\E,u}_s}
\newcommand{\XEup}{X^{\E,u}_p}
\newcommand{\XEur}{X^{\E,u}_r}
\newcommand{\XEuq}{X^{\E,u}_q}
\newcommand{\XEusplust}{X^{\E,u}_{s+t}}
\newcommand{\XEu}{X^{\E,u}}
\newcommand{\SEu}{S^{\E,u}}
\newcommand{\SEut}{S^{\E,u}_t}
\newcommand{\SEus}{S^{\E,u}_s}
\newcommand{\SEur}{S^{\E,u}_r}
\newcommand{\SEuzero}{S^{\E,u}_0}
\newcommand{\jEust}{\jmath^{\E,u}_{s,t}}
\newcommand{\jEurs}{\jmath^{\E,u}_{r,s}}
\newcommand{\jEut}{\jmath^{\E,u}_t}

\newcommand{\jEur}{\jmath^{\E,u}_r}
\newcommand{\phiEut}{\phi^{\E,u}_t}
\newcommand{\phiEu}{\phi^{\E,u}}
\newcommand{\FEut}{F^{\E,u}_t}
\newcommand{\FEur}{F^{\E,u}_r}
\newcommand{\FEu}{F^{\E,u}}
\newcommand{\kEu}{\kil(\E,u)}
\newcommand{\KEu}{\Kil^{\E,u}}
\newcommand{\KEut}{\Kil^{\E,u}_t}

\newcommand{\VEu}{V^{\E,u}}
\newcommand{\KEur}{\Kil^{\E,u}_r}
\newcommand{\phiEur}{\phi^{\E,u}_r}

\DeclareMathOperator{\Int}{Int}
\DeclareMathOperator{\adjointable}{a}
   \DeclareMathOperator{\Hausdorff}{H}

\DeclareMathOperator{\Ran}{Ran}
\DeclareMathOperator{\Lin}{Lin}

\DeclareMathOperator{\im}{Im}

\DeclareMathOperator{\supp}{supp}
\DeclareMathOperator{\esssupp}{ess-supp}
\DeclareMathOperator{\weaklim}{weak-lim}

\DeclareMathOperator{\sym}{sym}
\DeclareMathOperator{\ind}{ind}

\DeclareMathOperator{\Triv}{Triv}
\DeclareMathOperator{\Root}{Root}
\DeclareMathOperator{\Borel}{Borel}
\DeclareMathOperator{\T}{T}

\newcommand{\Linbar}{\ol{\Lin}}
\newenvironment{alist}
{

\begin{enumerate}}
{\end{enumerate}}

\newenvironment{rlist}
{

\begin{enumerate}}
{\end{enumerate}}

\newcommand{\la}{\langle}
\newcommand{\ra}{\rangle}

\begin{document}
	
	\title[Addits]{Additive units of product systems}
	
	%    Information for first author
	\author[Rajarama Bhat, Martin Lindsay and Mithun Mukherjee]{B.V. Rajarama Bhat}
	%    Address of record for the research reported here
	\address{Stat-Math Unit \\
Indian Statistical Institute \\
R.V. College Post \\
Bangalore-560059 \\
India}
	\email{bhat@isibang.ac.in}
	
	%    Information for second author
	\author[]{J.\ Martin Lindsay}
\address{Department of Mathematics \& Statistics
\\
Lancaster University
\\  Lancaster LA1 4YF \\ UK}
	\email{j.m.lindsay@lancaster.ac.uk}
	
	% Information for third author
	\author[]{Mithun Mukherjee}
	\address{School of Mathematics \\
 IISER Thiruvananthapuram \\
  CET Campus \\
  Kerala - 695016 \\
   India}
	\email{mithunmukh@iisertvm.ac.in}

	%    General info
	\subjclass[2010]{Primary 46L55;    % Noncommutative dynamical systems
                    Secondary 46C05,    % Hilbert and pre-Hilbert spaces: geometry and topology
                                        % (including spaces with semidefinite inner product)
	                           46L53}  %  Noncommutative probability and statistics
	%%%%%%%%%%%%%%%%%%%%%%%%%%
%TEMPORARY
%\date{January 1, 2001 and, in revised form, June 22, 2001.}
%\date{June 25 pm, 2017.}
%%%%%%%%%%%%%%%%%%%%%%%%%%%%
	
	%\dedicatory{This paper is dedicated to our advisors.}
	
	\keywords{Arveson systems,
            inclusion systems,
            quantum dynamics,
            completely positive semigroups,
            Cantor--Bendixson derivative,
            cluster construction}
	
	\begin{abstract}
 We introduce the notion of additive units, or `addits', of a pointed Arveson system,
 and demonstrate their usefulness through several applications.
 By a pointed Arveson system
 we mean a spatial Arveson system with a fixed normalised reference unit.
 We show that the
 addits form a Hilbert space whose codimension-one subspace of `roots' is
 isomorphic to the index space of the Arveson system,
 and that
 the addits generate the type I part of the Arveson system.
 Consequently
 the isomorphism class of the Hilbert space of addits is
 independent of the reference unit.
 The addits of a pointed inclusion system are shown to be
 in natural correspondence with the addits of the generated pointed product system.
 The theory of amalgamated products
 is developed using addits and roots,
 and an
 explicit formula for
 the amalgamation of pointed Arveson systems
 is given,
 providing a new proof
 of its independence of the particular reference units.
(This independence
justifies the terminology
 `spatial product' of spatial Arveson systems).
  Finally
 a cluster construction for
 inclusion subsystems of an Arveson system is introduced
 and
 we demonstrate
 its correspondence with the action of the Cantor--Bendixson derivative in the context of
 the random closed set approach to product systems
 due to Tsirelson and Liebscher.
	\end{abstract}
	
	\maketitle

\vspace*{-11cm}
\begin{flushright}
\ti{\small To appear in:
\\
Transactions of the American Mathematical Society}
\end{flushright}
\vspace{11cm}

%%%%%%%%%%%%%%%%%%%%%%%%%%%%%%%%%%%      TABLE OF CONTENTS
 \tableofcontents

%%%%%%%%%%%%%%%%%%%%%%%%%%%%%%%%%%%%%%%%%%%%% New section

\section*{Introduction}
\label{section: introduction}

A basic goal of the study of quantum dynamics is the classification of
$E_0$-semigroups, that is suitably continuous
one-parameter semigroups of unital *-endomorphisms of $B(\Hil)$,
the algebra of bounded operators on a separable Hilbert space $\Hil$
(\cite{Arv-noncommutative}).
Each $E_0$-semigroup is associated to
an Arveson system, that is a suitably measurable
one-parameter family of separable Hilbert spaces $\E = (\E_t)_{t>0}$
enjoying associative identifications $\E_{s+t} \simeq \E_s\otimes \E_t$
via unitary operators,
and conversely,
to each such Arveson system there is an associated $E_0$-semigroup.
If cocycle conjugate $E_0$-semigroups are identified,
and isomorphic Arveson systems are too,
then these associations are rendered mutually inverse
(\cite{Arv-continuous},\cite{Arv-converse};
see also~\cite{Lie-random}, and~\cite{Skeide-simple}).

A unit of an Arveson system is a nonzero measurable section $(u_s)_{s>0}$,
which has the continuous factorisation property:
$u_{s+t}=u_s\otimes u_t$,
and Arveson systems are classified
into type I, type II and (nonspatial or) type III,
according to whether their set of units respectively,
generates the system, is nonempty but fails to generate the system, or is empty.
Spatial Arveson systems have an associated index space;
this is a separable Hilbert space constructed from the set of units
whose dimension is called the index of the system.
The index is an isomorphism invariant,
and is additive under the tensor product operation on Arveson systems.

For type I Arveson systems the index is a complete invariant (\cite{Arv-continuous})
and,
for each separable Hilbert space $\kil$
there is a paradigm type I system with index equal to $\dim \kil$,
namely the Fock Arveson system $\Fockk$ (\cite{Arv-continuous});
this is
described in the appendix.
The isomorphism classes of type II and type III systems are both
known to be uncountable
(\cite{Pow-nonspatial},\cite{Pow-new-example},\cite{Tsi-non-isomorphic},\cite{Tsi-coloured}).
There is currently a lack of good invariants to distinguish these, and
their classification is far from complete.
Tsirelson has shown
measure types of random sets, and generalised Gaussian processes,
to be fertile sources of type II systems
(\cite{Tsi-three-questions},\cite{Tsi-Warren's-noise});
Liebscher has made a systematic study of Tsirelson's examples.
To every product subsystem of an Arveson system $\E$
there corresponds a commuting family of orthogonal projections
satisfying evolution and adaptedness relations,
and the von Neumann algebra generated by them uniquely
determines a (probability) measure type
of random closed subsets of the unit interval.
The measure types are stationary and factorising over disjoint intervals, and
provide an isomorphism invariant for the Arveson system (\cite{Lie-random}).

Completely positive contraction semigroups on operator *-algebras are called
quantum dynamical semigroups.
For a separable Hilbert space $\Hil$,
every unital quantum dynamical semigroup on $B(\Hil)$
dilates to an $E_0$-semigroup,
and the minimal dilation is unique up to cocycle conjugacy;
this provides an approach to
the understanding of quantum dynamics (\cite{Bha-minimal}).
For $E_0$-semigroups on
$C^*$- and $W^*$-algebras,
one may associate product systems of Hilbert modules
(\cite{MuS-Markov},\cite{ShS-subproduct},\cite{Ske-classification}).
Much of the theory of Arveson systems and $E_0$-semigroups on $B(\Hil)$
carries over to
product systems of Hilbert modules and $E_0$-semigroups on $B^{\adjointable}(E)$,
the algebra of adjointable operators on a Hilbert module $E$.
However there is no tensor product operation for
product systems of Hilbert modules.
For pointed product systems of Hilbert modules,
that is systems with a fixed normalised reference unit,
Skeide overcame this by introducing a notion of spatial product
(\cite{Ske-index}).
In the spatial product,
units are identified and the index is again additive.

By a pointed Arveson system we mean
a spatial Arveson system together with a fixed normalised reference unit.
For pointed Arveson systems $(\E,u)$ and $(\F,v)$,
Skeide's spatial product may be identified with
$\E \ot v \bigvee u \ot\F$,
the product subsystem of the tensor product Arveson system $\E \ot \F$
generated by $\E \ot v$ and $u \ot\F$.
This raises the natural question, is this necessarily all of $\E \ot \F$?
Powers answered this in the
negative, by solving the corresponding equivalent problem
for $E_0$-semigroups using his `sum construction'
(\cite{Pow-addition});
see also~\cite{Ske-commutant},~\cite{BLS-subsystems} and~\cite{Ske-CPD}).
Motivated by this question,
the amalgamated product via a contractive morphism of Arveson systems
(which are not necessarily spatial)
was introduced in~\cite{BhM-inclusion}
(see Section~\ref{section: amalgamated products}).
This generalises the spatial product of pointed Arveson systems
since the latter may be viewed as the amalgamated product via
the morphism defined through Dirac dyads of the normalized units.
(It also answers Powers' problem for
the Powers sum arising from not-necessarily-isometric intertwining
semigroups.)
A priori the spatial product may depend on the reference units.
Since,
as Tsirelson has shown,
the automorphism group of an Arveson system may not act transitively
on its set of units
(\cite{Tsi-automorphism})
the answer to this dependency question is not obvious.
It was settled in the negative in
\cite{BLMS-intrinsic}, see also \cite{Lie-spatial}.
Our work yields another proof of this fact.

In this paper we introduce and systematically exploit
the notions of addit and root,
for pointed Arveson systems.
We also introduce a cluster construction
for product subsystems of an Arveson system;
on the one hand the construction provides
a new way of obtaining the type I part of a spatial Arveson system
(Theorem~\ref{thm: 5.5}),
on the other hand it reflects the extraction of the derived sets
of random closed subsets of the unit interval in Liebscher's correspondence
(Theorem~\ref{thm: 5.2}).
Whereas Liebscher's work  heavily relies on
  direct integral constructions
  and
  the measure theory of random sets,
 by contrast
 our cluster construction
 (Definition~\ref{defn: 6.2} \emph{et seq}) is done explicitly by
  elementary Hilbert space means,
 via an inclusion subsystem.

The structure of the paper is as follows.
In
Sections~\ref{section: product systems}
and~\ref{section: inclusion systems},
we give a brief overview of the basic theory of
product systems, Arveson systems and inclusion systems,
and set out the notations and terminology used in the paper.
This includes an important implication of Liebscher's work (Theorem~\ref{thm: always}).
An appendix describes the paradigm case of Fock Arveson systems
$\Fockk$, for a separable Hilbert space $\kil$,
and introduces the `Guichardet picture' for these.
 In Section~\ref{section: addits Arveson}
 addits and roots are defined.
These are additive counterparts to units;
 roots are addits which are orthogonal to the reference unit.
 Addits comprise a Hilbert space with roots occupying a codimension-one subspace.
 The roots of the pointed Arveson systems $(\Fockk, \Vack)$, in which $\Vack$ is the vacuum unit,
 are shown to be indexed by the elements of $\kil$ itself,
 via an isometric isomorphism (Proposition~\ref{propn: roots of Fock}).
 From this we show that,
 for any normalised unit,
 the type I part of a spatial Arveson system is generated by the unit together with its roots,
 and the dimension of the Hilbert space of roots equals the index of the Arveson system
 (Theorem~\ref{thm: 3.7}).
 Thus the isomorphism class of the Hilbert space of addits of a pointed Arveson system $(\E,u)$
 is independent of
 the choice of unit $u$ of the spatial Arveson system $\E$.
 In Section~\ref{section: addits inclusion},
 we extend the notions of addit and root to pointed inclusion systems $(E,u)$,
 and establish a
 natural bijection
 between the addits of such a system and the addits of $(\E, \wh{u})$
where $\E$ is the generated (algebraic) product system and $\wh{u}$ is the normalised unit
obtained from $u$ by `lifting' (Proposition~\ref{propn: 3.12}).
  The behaviour of roots under
  amalgamated products of  both spatial and pointed Arveson systems
  is studied
   in Section~\ref{section: amalgamated products}.
   In that section
   we give an explicit formula for the amalgamated product of pointed Arveson systems
   (Theorem~\ref{thm: 4.2}) which
 provides another proof of its independence of the reference units,
 and thus also
 of the fact that,
   up to cocycle conjugacy,
    the Powers sum of $E_0$-semigroups is independent of the choice of intertwining isometries.
In
Section~\ref{section: cluster construction}
we describe our cluster construction for
subsystems $\F$ of an Arveson system $\E$;
we also summarise the relevant theory of hyperspaces.
When $\E$ is spatial,
and $\F$ is generated by a normalised unit,
the cluster
is shown to be the type I part of $\E$ (Theorem~\ref{thm: 5.5}).
Finally,
extending part of Proposition 3.33 of~\cite{Lie-random},
  we show that
the measure type corresponding to a subsystem
and the measure type of its cluster are
precisely related
via the Cantor--Bendixson derivative (Theorem~\ref{thm: 5.2}).

\emph{Some notational conventions}.
For Hilbert space vectors $u \in \Hil$ and $x \in \Kil$,
$\dyad{x}{u}$ denotes the bounded operator $\Hil \to \Kil$, $v \mapsto \ip{u}{v} x$
(inner products being linear in their second argument).
For a subset $A$ of the domain of a vector-valued function $g$,
$g_A$ denotes the function which equals $g$ except that it takes the value $0$ outside $A$
(\emph{cf}. indicator function notation).
We use $\mathcal{P}$ to denote power set,
and $\subset \subset$ for subset of finite cardinality.

%%%%%%%%%%%%%%%%%%%%%%%%%%%%%%%%%%%%%%%%%%%%% New section

\section{Product systems}
%\label{Product-inclusion}
\label{section: product systems}

In this section we
briefly recall the basic concepts of continuous product systems of Hilbert spaces,
and thereby establish our basic notations.
Key references are Arveson's monograph (\cite{Arv-noncommutative}) and Liebscher's memoir (\cite{Lie-random}).

\begin{definition}
An (\emph{algebraic}) \emph{product system} $\E$ consists of
a family of Hilbert spaces $(\Et)_{t>0}$
with associated unitary \emph{structure maps}
\[
\BEst: \Esplust \to \Es \ot \Et
\qquad(s,t > 0),
\]
satisfying the natural consistency conditions
\[
( \IEr \ot \BEst ) \BErsplust =
( \BErs \ot \IEt ) \BErplusst
\qquad(r,s,t > 0)
\]
where
$\IEs := I_{\Es}$ ($s>0$).
It is called an \emph{Arveson system} if
each fibre $\Et$ is separable and the system is endowed with measurable structure:
the families $(\Et)_{t>0}$ and
$( \BEst )_{s,t>0}$ are both `measurable'.
\end{definition}

\begin{remarks}
(i)
In the literature, the structure maps are usually taken to be the adjoints
$\WEst = ( \BEst )^* : \Es \ot \Et \to \Esplust$.
Here we use the equivalent $B$'s instead
in order to maintain consistency with inclusion systems (defined below).

(ii)
For the precise meaning of measurability meant here, we refer to~\cite{Arv-noncommutative}
and
the essentially equivalent formulation given in~\cite{Lie-random}.

(iii)
Frequently one supresses the structure maps and identifies
$\Esplust$ and $\Es \ot \Et$,
or writes $x \bfcdot y$ for the preimage in $\Esplust$ of $x \ot y$
when $x \in \Es$ and $y \in \Et$.

(iv)
If $\dim \Et = 1$ for each $t>0$
then a choice of unit vector $u_t \in \Et$ for each $t>0$
reduces the consistency condition to the multiplier relation
\[
m(s,t) m(r, s+t) = m(r,s) m(r+s,t)
\]
for the map $m: \Real_{>0} \times \Real_{>0} \to \mathbb{T}$ given by
$m(s,t) u_s \ot u_t = \BEst u_{s+t}$.
\end{remarks}

\begin{definition}
\label{defn: 1.2}
Let $\E$ be a product system and let $T > 0$.
The family of unitary operators
$\UET = ( \UET_t )_{t \in \Real}$
on $\ET$
defined by periodic extension of the prescription
\[
\UETt =
\left\{
 \begin{array}{ll}
 \IET
 & \text{ if }
 t=0
 \\
 ( \BEtTminust )^* \Pi^T_t \BETminustt
 & \text{ if }
 0 < t < T
 \end{array}, \right.
\]
in which
$\Pi^T_t$ denotes the tensor flip $\ETminust \ot \Et \to \Et \ot \ETminust$,
is
called the \emph{unitary flip group on} $\ET$.
\end{definition}
It is easily verified  that
$\UET = ( \UET_t )_{t \in \Real}$
forms a one-parameter group.

\begin{theorem} [\cite{Lie-random}, Theorem 7.7]
\label{thm: Lieb 7.7}
Let $\E$ be a product system and let $\tau > 0$.
Then the following are equivalent\tu{:}
\begin{rlist}
\item
$\E$ is an Arveson system with respect to some measurable structure.
\item
For all $t>0$,
$\E_t$ is separable, and
for all $T \in \, ]0,\tau[$\,,
$\UET$ is strongly continuous.
\end{rlist}
\end{theorem}

Let $\E$ be a product system and suppose that, for each $t>0$,
$\Ft$ is a closed subspace of $\Et$ and that, for each $s,t>0$,
$\BEst (\Fsplust ) = \Fs \ot \Ft$.
Then $\F = ( \Ft )_{t>0}$ is a product system with structure maps
$\BFst : \Fsplust \to \Fs \ot \Ft$ ($s,t>0$)
given by compression of the structure maps of $\E$.
Such systems are called \emph{product subsystems of} $\E$.

The following automatic measurability result is a
significant consequence of Liebscher's approach to product systems.
Note that in his approach the parameter set of an Arveson system $\E$
is extended to $\Rplus$, with $\E_0 = \Comp$.

\begin{theorem}
\label{thm: always}
Let $\F$ be a product subsystem of an Arveson system $\E$.
Then $\F$ is an Arveson subsystem, in other words
the measurable structure of $\E$ induces measurable structure on $\F$.
\end{theorem}

\begin{proof}
Let $( e^n )_{ n\ges 1}$ be a family of sections of $\E$ determining
its measurable structure. We must show that
\begin{alist}
\item
the sections
$\big(  P_t e^n_t \big)_{ n\ges 1}$ are measurable, and
\item
the family of operators
$\big(  W^{\F}_{s,t} :=
V^*_{s+t} W^{\E}_{s,t} ( V_s \ot V_t ) \big)_{ s,t \ges 0}$ is measurable,
\end{alist}
for the inclusion operators $V_t: \F_t \to \E_t$
and orthogonal projections $P_t := V_t V_t^* = P_{\F_t}$.
Without loss of generality we may suppose that
$(  e^n_t )_{ n\ges 1}$ is an orthonormal basis of $\E_t$ for all $t > 0$,
moreover it suffices to prove
(a) and (b) for
these for $s$  and $t$ ranging over $]0,1[$ (see~\cite{Lie-random}).

Let $\big( P^{\F}_{s,t} \big)_{0 \les s,t \les 1}$ be
the strongly continuous family of orthogonal projections in $B( \E_1 )$
defined in~\eqref{eqn: PF} below, and set $e := e^1$.
By Parseval's identity,
the measurability of $t \mapsto e^p_t \cdot e^q_{1-t}$ ($p,q \in \Nat$),
 and the strong continuity of $t \mapsto P^{\F}_{0,t}$
it follows that, for all $l, m \ges 1$ and $t \in [0,1]$,
\begin{align*}
\ip{ e^l_t }{ P_t e^m_t }
&
=
\ip{ e^l_t \cdot e_{1-t}}{ P^{\F}_{0,t} ( e^m_t \cdot e_{1-t} ) }
\\
&
=
\sum_{p,q \ges 1}
\ip{ e^l_t \cdot e_{1-t} }{ e^p_{1} } \,
\ip{ P^{\F}_{0,t} e^p_1 }{ e^q_{1} } \,
\ip{ e^q_1 }{ e^m_t \cdot e_{1-t} }
\end{align*}
which is now manifestly measurable in $t$. This proves (a).
By another application of Parseval's identity, we see that
\begin{align*}
\ip{ V^*_{s+t} e^l_{s+t} }{ W^{\F}_{s,t} ( V^*_s e^m_s \ot V^*_t e^n_t ) }
&
=
\ip{ e^l_{s+t} }{ W^{\E}_{s,t} ( P_s e^m_s \ot P_t e^n_t ) }
\\
&
=
\sum_{p,q \ges 1}
\ip{ e^l_{s+t} }{ W^{\E}_{s,t} ( e^p_s \ot e^q_t ) } \,
\ip{ e^p_{s} }{ P_s e^m_s } \,
\ip{ e^q_{t} }{ P_t e^n_t }
\end{align*}
for
$l, m, n \ges 1$ and $s, t \in [0,1]$,
so (b) follows from (a).
\end{proof}

Given two product subsystems $\Eone$ and $\Etwo$ of a product system $\E$,
the smallest product system of $\E$ containing $\Eone$ and $\Etwo$ is denoted
$\Eone \bigvee \Etwo$.
Thus, by Theorem~\ref{thm: always}, if $\E$ is an Arveson system then
$\Eone \bigvee \Etwo$ is an Arveson subsystem of $\E$.

\begin{definition}
Let $\E$ and $\F$ be product systems.
A family of bounded operators $\phi = ( \phi_t: \E_t \to \F_t )_{t>0}$ is a
\emph{morphism of product systems} if it satisfies
\[
\BFst \, \phi_{s+t} = ( \phi_s \ot \phi_t ) \BEst
\qquad
(s,t > 0)
\]
and the quasicontractivity condition
$
e^{-kt} \norm{\phi_t} \les 1
$
($t>0$), for some $k \in \Real$;
it is an \emph{isomorphism} if each $\phi_t$ is unitary.
A \emph{morphism of Arveson systems} is
a morphism of the underlying product systems
which consists of a measurable family of operators.
\end{definition}

\begin{theorem}
[\cite{Lie-random}, Corollary 7.16]
\label{thm: 7.16}
Let $\phi: \E \to \F$ be an isomorphism of product systems.
Suppose that $\E$ and $\F$ are Arveson systems.
Then $\phi$
and $\phi^{-1}$ are measurable, and thus
$\phi$ is an isomorphism of Arveson systems.
\end{theorem}

\begin{definition}
Let $\E$ be an Arveson system.
A \emph{unit} of $\E$ is a nonzero measurable section of $\E$ satisfying
\[
u_{s+t} = u_s \bfcdot u_t
\qquad(s,t > 0)\tu{;}
\]	
it is \emph{normalised} if it satisfies $\norm{u_t} = 1$ ($t>0$).
The collection of units of $\E$, respectively normalised units of $\E$,
is denoted $\UnitE$, respec. $\UnitEnorm$,
and $\E$ is called \emph{spatial} if $\UnitE \neq \emptyset$.

The \emph{type I part of} $\E$, denoted $\EI$,
is the smallest product subsystem of $\E$ containing all the units of $\E$,
and $\E$ is said to be \emph{of type I} if $\EI$ is $\E$ itself.
Thus, for a spatial Arveson system $\E$,
\[
\EIT =
\Linbar
\big\{
u^1_{t_1} \bfcdot \, \cdots \, \bfcdot \, u^n_{t_n}:
n \in \Nat, u^1,  \cdots , u^n \in \UnitE, {\bf t} \in J^{(n)}_T
\big\}
\qquad(T>0)
\]
where $J^{(n)}_T  := \{ {\bf t} \in ( \Real_{>0} )^n: \sum t_i = T \}$.
\end{definition}

Let $\E$ be a spatial Arveson system.
For each $u,v \in \UnitE$,
the function $t \mapsto \ip{u_t}{v_t}$ is measurable and
satisfies Cauchy's multiplicative functional equation $f(s+t) = f(s) f(t)$,
and so there is $\gamma(u,v) \in \Comp$ such that
$\ip{u_t}{v_t} = e^{t \gamma(u,v)}$ ($t>0$).
The resulting map $\gamma: \UnitE \times \UnitE \to \Comp$ is called the \emph{covariance function} of $\E$.
It is conditionally positive definite:
$\sum \ol{\lambda_i} \lambda_j \gamma(u^i, u^j) \ges 0$
for $n \in \Nat$,
$u^1, \cdots , u^n \in \UnitE$ and
$\lambda_1, \cdots , \lambda_n \in \Comp$ satisfying $\sum \lambda_i = 0$.
It follows that the prescription
\[
\ip{f}{g} := \sum_{u,v \in \UnitE} \gamma(u,v) \ol{f(u)} f(v)
\]
defines a nonnegative sesquilinear form on the vector space
\[
V :=
\Big\{
f: \UnitE \to \Comp \, \Big| \ \supp f \subset\subset \UnitE, \sum\nolimits_{u \in \UnitE} f(u) = 0
\Big\}.
\]
Quotienting out by the null space $\{ f \in V: \ip{f}{f} = 0 \}$
and completing yields a Hilbert space $\kil(\E)$, called the \emph{index space of} $\E$;
its dimension, denoted $\ind \E$, is called the \emph{index of} $\E$.

The index is an isomorphism invariant for Arveson systems:
if $\Eone \cong \Etwo$ then $\kil(\Eone) \cong \kil(\Etwo)$.

\begin{example}
Our notations for Fock Arveson systems are given in the appendix.
The covariance function of
the Fock Arveson system
$\Fockk$ is given by
\[
\gamma \big( (e^{\lambda t} \ve^c_t )_{t>0}, (e^{\mu t} \ve^d_t )_{t>0} \big)
=
\ol{\lambda} + \mu + \ip{c}{d}
\qquad
( c, d \in \kil, \, \lambda, \mu \in \Comp).
\]
These Arveson systems are of type I and satisfy
\[
\kil( \Fockk ) \cong \kil.
\]
Thus $\Fockkone \cong \Fockktwo$ implies $\kil_1 \cong \kil_2$.
Conversely,
$\E^I \cong \Fock^{\kil(\E)}$
for any Arveson system $\E$.
\end{example}

The following notion
plays a central role in this paper, from Section~\ref{section: addits Arveson} onwards.

\begin{definition}
\label{defn: pointed}
A \emph{pointed Arveson system}
is an ordered pair $(\E, u)$
consisting of a spatial Arveson system $\E$ and a fixed normalised unit $u$,
which we refer to as the \emph{reference unit}.
\end{definition}

\begin{remarks}
Our terminology is a refinement of Liebscher's (in~\cite{Lie-spatial});
his is in conflict with the now-common use of the term spatial Arveson system
(as defined above).

There is an obvious notion of \emph{isomorphism} for
pointed Arveson systems.

By means of Fock--Weyl operators (see the appendix),
it is easily seen that,
for a type I Arveson system $\E$,
the family of pointed systems $\{ (\E, u): u \in \UnitEnorm \}$
are all isomorphic.
However,
in view of a theorem of Tsirelson (\cite{Tsi-automorphism}),
this need not be true for
type II Arveson systems.
\end{remarks}

%%%%%%%%%%%%%%%%%%%%%%%%%%%%%%%%%%%%%%%%%%%%% New section

\section{Inclusion systems}
%\label{Product-inclusion}
\label{section: inclusion systems}

In this section we introduce notations for inclusion systems and recall their basic theory.
We also describe the Fock inclusion systems.
Inclusion systems are defined like product systems
except that their structure maps are only required to be isometric.
They arise very often in quantum dynamics.
For instance, the product system associated with a completely positive semigroup
on the algebra of bounded operators on a separable Hilbert space is in fact
the product system generated by an inclusion system derived from the semigroup
(\cite{BhS-tensor},\cite{MuS-Markov},\cite{Mar-CP},\cite{ShS-subproduct},\cite{BhM-inclusion}).
Our basic reference is~\cite{BhM-inclusion}, where inclusion systems were introduced.
Shalit and Solel also studied them, in a more abstract setting,
under the name \emph{subproduct systems} (\cite{ShS-subproduct}).

  \begin{definition}
An \emph{inclusion system} $E$ is a family of Hilbert spaces $(E_t)_{t>0}$
together with isometric \emph{structure maps}
$\betaEst: E_{s+t} \to E_s \ot E_t$ ($s,t>0$) satisfying
\[
( I^E_r \ot \betaEst ) \, \betaErsplust = ( \betaErs \ot I^E_t ) \, \betaErplusst
\qquad
(r,s,t > 0),
\]
where
$I^E_s := I_{E_s}$ ($s>0$).
\end{definition}

\begin{remark}
Thus a product system is an inclusion system whose structure maps are unitary.
\end{remark}

 \begin{definition}
Let $E$ be an inclusion system.
If, for all $t>0$, $F_t$ is a closed subspace of  $E_t$,
and, for all $s,t>0$, $\betaEst (F_{s+t}) \subset F_s \ot F_t$,
then the isometries
$\betaFst: F_{s+t} \to F_s \ot F_t$ ($s,t>0$) induced by compression render
$F$ an inclusion system.
Such systems are called \emph{inclusion subsystems of} $E$.
\end{definition}

We now define the product system generated by an inclusion system.
It is an inductive limit construction.

\begin{notation}
For $T>0$, set
\[
J_T := \bigcup_{n=1}^\infty J^{(n)}_T
\ \text{ where } \
J^{(n)}_T := \Big\{ {\bf t} \in ( \Real_{>0})^n: \, \sum t_i = T \Big\},
\]
and for $S,T>0$, $m,n \in \Nat$, ${\bf s } \in J^{(m)}_S$ and ${\bf t} \in J^{(n)}_T$,
set
\[
{\bf s} \smile {\bf t} := ( s_1, \cdots , s_m, t_1, \cdots , t_n) \in J^{(m+n)}_{S+T}.
\]
\end{notation}

A partial order on $J_T$ is defined as follows.
For ${\bf r} \in J^{(m)}_T$ and ${\bf s} \in J_T$,
\[
{\bf s} \ges {\bf r}
\ \text{ if } \
{\bf s} = {\bf r}_1 \smile \cdots \smile {\bf r}_m
\ \text{ where } \
 {\bf r}_i \in J_{r_i}
 \text{ for } i = 1, \cdots , m.
\]
Thus $(T) \les {\bf t}$ for all ${\bf t} \in J_T$.
The partially ordered set $J_T$ is directed:
\[
\forall_{{\bf r, s} \in J_T} \exists_{{\bf t} \in J_T}: \,
{\bf t} \ges {\bf r} \text{ and } {\bf t} \ges {\bf s}.
\]

Let $E$ be an inclusion system and fix $T > 0$ for now.
For ${\bf t} \in J^{(n)}_T$, set
$\Ebft := E_{t_1} \ot \cdots \ot  E_{t_n}$,
thus $E_{(T)} = E_T$.
Define isometries
$\big( \betaEbfst: \Ebfs \to \Ebft \big)_{{\bf s} \les {\bf t} \text{ in } J_T}$
as follows: for $p \in \Nat$ and ${\bf r} \in J^{(p)}_R$ set
\[
\betaERbfr =
\left\{
 \begin{array}{ll}
 I^E_R
 & \text{ if }
 {\bf r} = (R)
 \\
 \big( I^E_{r_1} \ot \cdots \ot I^E_{r_{p-2}} \ot \beta^E_{r_{p-1}, r_p} \big)
 \cdots
 \big( I^E_{r_1} \ot \beta^E_{r_2, r_3 + \cdots + r_p} \big) \beta^E_{r_1, r_2  + \cdots + r_p}
 & \text{ otherwise }
 \end{array}, \right.
\]
and for ${\bf s} \les {\bf t}$ with ${\bf s} \in J^{(m)}_ T$
and
${\bf t} = {\bf s}_1 \smile \cdots \smile {\bf s}_m$,
\[
\betaEbfst := \beta^E_{s_1, {\bf s}_1} \ot \cdots \ot \beta^E_{s_m, {\bf s}_m}.
\]
Thus $\betaEbfss = I^E_{{\bf s}} := I_{E_{{\bf s}}}$.

For $T>0$,
$\big( (\Ebft)_{{\bf t} \in J_T}, ( \betaEbfrs )_{{\bf r} \les {\bf s} \in J_T} \big)$
forms an inductive system of Hilbert spaces:
\[
\betaEbftt = I^E_{{\bf t}} \quad ({\bf t} \in J_T) \ \text{ and } \
\betaEbfst \, \betaEbfrs = \betaEbfrt \quad ({\bf r} \les {\bf s} \les {\bf t} \text{ in } J_T).
\]
Let
$\big( \E_T, ( \iEbft: \Ebft \to \ET )_{{\bf t} \in J_T} \big)$
denote its inductive limit.
For ease of reference, we list its key properties next.
\begin{rlist}
\item
\emph{Minimality}.
$\ET$ is a Hilbert space satisfying $\ET = \vee_{{\bf t} \in J_T} \Ran \iEbft$.
\item
\emph{Isometry}.
$\iEbft$ is an isometry (${\bf t} \in J_T$) and
$\iEbfs \circ \betaEbfrs = \iEbfr$
(${\bf r} \les {\bf s} \text{ in } J_T$).
\item
\emph{Subnet property}.
For any $K \subset J_T$ such that
$
\forall_{{\bf s} \in J_T} \exists_{{\bf t} \in K}\!: \,
{\bf t} \ges {\bf s}
$,
the inductive limit of
$\big( (\Ebft)_{{\bf t} \in K}, ( \betaEbfrs )_{{\bf r} \les {\bf s} \text{ in } K} \big)$
equals
$\big( \E_T, ( \iEbft: \Ebft \to \ET )_{{\bf t} \in K} \big)$.
\item
\emph{Universal property}.
For $K \subset J_T$ as in (iii),
and any family of Hilbert space isometries
$( \jmath_{{\bf t}} : \Ebft \to \Hil )_{{\bf t} \in K}$
satisfying
$
\jmath_{{\bf s}} \circ \betaEbfrs = \jmath_{{\bf r}}
$
(${\bf r} \les {\bf s} \text{ in } K$),
there is a unique isometry
$\jmath: \ET \to \Hil$ such that $\jmath_{{\bf t}} = \jmath \circ \iEbft$ (${\bf t} \in K$).
\end{rlist}

Now let $R,S>0$ and set
$J_R \smile J_S := \{ {\bf r} \smile {\bf s}: {\bf r} \in J_R, {\bf s} \in J_S \}$.
For
${\bf t} \in J_{R+S}$
there are
${\bf r} \in J_R$ and  ${\bf s} \in J_S$
such that
${\bf r} \smile {\bf s} \ges {\bf t}$.
Therefore,
by the subnet property (iii),
the inductive limit of
$\big( (\Ebft)_{{\bf t} \in J_R \smile J_S}, ( \betaEbfuv )_{{\bf u} \les {\bf v} \text{ in } J_R \smile J_S} \big)$
equals
$\big( \E_T, ( \iEbft: \Ebft \to \ET )_{{\bf t} \in J_R \smile J_S} \big)$
where $T = R+S$.
For
${\bf r}, {\bf r}' \in J_R$ and ${\bf s}, {\bf s}' \in J_S$
such that
$
{\bf r} \smile {\bf s} \les {\bf r}' \smile {\bf s}'
$,
necessarily
$
{\bf r}  \les {\bf r}'
$
and
$
{\bf s} \les  {\bf s}'
$
so
\[
( \imath^E_{{\bf r}'} \ot \imath^E_{{\bf s}'} ) \circ
\beta^E_{ {\bf r} \smile {\bf s}, {\bf r}' \smile {\bf s}' }
=
( \imath^E_{{\bf r}'} \circ \beta^E_{ {\bf r}, {\bf r}' } )
\ot
( \imath^E_{{\bf s}'} \circ \beta^E_{ {\bf s}, {\bf s}' } )
=
\imath^E_{{\bf r}} \ot \imath^E_{{\bf s}}.
\]
The family
$\big( \wt{\imath}_{ {\bf t} } :
\Ebft \to \ER \ot \ES )_{{\bf t} \in J_R \smile J_S} \big)$,
in which
$
\wt{\imath}_{ {\bf r} \smile {\bf s} } := \iEbfr \ot \iEbfs
$,
satisfies
$
\wt{\imath}_{ {\bf t}' } \circ \beta^E_{ {\bf t}, {\bf t}' }
=
 \wt{\imath}_{ {\bf t} }
$
for
$
{\bf t} \les {\bf t}'
$
in
$
J_R \smile J_S
$.
Therefore,
by the universal property (iv),
there is a unique isometry
$\BERS: \ERplusS \to \ER \ot \ES$
such that
$
\wt{\imath}_{ {\bf t} }
=
\BERS \circ \iEbft
$
(${\bf t} \in J_R \smile J_S$), equivalently
$
\iEbfr \ot \iEbfs = \BERS \circ \imath^E_{ {\bf r} \smile {\bf s} }
$
($
{\bf r}\in J_R$, ${\bf s} \in J_S
$).
It follows from the minimality property (i) that
$\Ran \BERS = \ER \ot \ES$,
so
$\BERS$ is unitary.
It is now easily verified that, for $R,S,T >0$,
$
{\bf r}\in J_R$, ${\bf s} \in J_S$ and ${\bf t} \in J_T
$,
\[
( \BERS \ot \IET ) \BERplusST \circ \imath^E_{ {\bf r} \smile {\bf s} \smile {\bf t}}
\ \text{ and } \
( \IER \ot \BEST ) \BERSplusT \circ \imath^E_{ {\bf r} \smile {\bf s} \smile {\bf t}}
\]
both equal
$\iEbfr \ot \iEbfs \ot \iEbft$.
Since
$
\bigcup_{ {\bf u} \in J_R \smile J_S \smile J_T } \Ran \imath^E_{{\bf u}}
$
is total in
$\ERplusSplusT$,
it follows that
$
( \BERS \ot \IET ) \BERplusST
=
( \IER \ot \BEST ) \BERSplusT
$
($R,S,T>0$).
In the above notations, we have established the following theorem.

\begin{theorem}
[\cite{BhM-inclusion}, Theorem 5]
The family $( \ET )_{T>0}$ defined above forms a product system with respect to the structure maps
$( \BEST )_{S,T>0}$.
\end{theorem}

As mentioned above, this is called the \emph{product system generated by} $E$.

\begin{theorem}
Let $\E$ be a product system and let $F$ be an inclusion subsystem.
Then the product system generated by $F$ may be viewed as a product subsystem of $\E$.
\end{theorem}

\begin{proof}
Let $\F$ be the product system generated by $F$.
We need to obtain an isometric morphism of product systems $\jmath: \F \to \E$.

Let $T>0$.
Consider the isometries
$
( \BETbft )^*|_{\Fbft} : \Fbft \to \ET
$
(${\bf t} \in J_T$).
For
${\bf r}  \les {\bf s}$ in $J_T$,
\[
( \BETbfs )^*|_{\Fbfs} \circ \betaFbfrs
=
( \BETbfs )^* {\BEbfrs}|_{\Fbfr}
=
( \BETbfr )^*|_{\Fbfr}
\]
Therefore, by the universal property (iv),
there is a unique isometry
$\jmath_T: \FT \to \ET$,
such that
$\jmath_T \circ \iFbft = ( \BETbft )^*|_{\Fbft}$
for the canonical maps
$\iFbft: \Fbft \to \FT$ (${\bf t} \in J_T$).

Now fix $S,T>0$.
In view of the identity
\[
B^{\mathcal{E}}_{S+T, {\bf s} \smile {\bf t}}
=
( \BESbfs \ot \BETbft ) \BEST
\qquad
( {\bf s} \in J_S, {\bf t} \in J_T ),
\]
which is not hard to verify,
\[
\BEST \circ \jmath_{S+T} \circ \imath^F_{ {\bf s} \smile {\bf t} }
=
\BEST \circ ( B^{\mathcal{E}}_{S+T, {\bf s} \smile {\bf t}} )^*|_{ F_{ {\bf s} \smile {\bf t} } }
=
( \BESbfs \ot \BETbft )^*|_{ F_{ {\bf s} } \ot F_{ {\bf t} } }
\]
and so, since

\[
( \jmath_S \ot \jmath_T ) \circ \BFST \circ \imath^F_{ {\bf s} \smile {\bf t} }
=
( \jmath_S \ot \jmath_T ) \circ ( \iFbfs \ot \iFbft )
=
( \BESbfs )^*|_{ F_{ {\bf s} } }
\ot
( \BETbft )^*|_{ F_{ {\bf t} } },
\]
the operators
$\BEST \circ \jmath_{S+T}$ and
$( \jmath_S \ot \jmath_T ) \circ \BFST$
agree on the set
$
\bigcup_{ {\bf u} \in J_S \smile J_T } \Ran \imath^F_{{\bf u}}
$
which is total in
$\FSplusT$.
It follows that the family of isometries $( \jmath_T: \FT \to \ET )_{T>0}$
forms a morphism of product systems, as required.
\end{proof}

\begin{definition}
Let $E$ and $F$ be inclusion systems.
A \emph{morphism} $E \to F$ is
a family of bounded operators $A = ( A_t: E_t \to F_t )_{t>0}$
satisfying the compatibility condition
\begin{equation}
\label{eqn: morphism incl}
A_{s+t} = ( \betaFst )^* (A_s \ot A_t ) \betaFst
\qquad
(s,t>0),
\end{equation}
and the quasicontractivity condition
$e^{-kt} \norm{A_t} \les 1$ ($t>0$)
for some $k \in \Real$.
It is called a \emph{strong morphism}
if~\eqref{eqn: morphism incl} is strengthened to
$
\betaFst A_{s+t} =  (A_s \ot A_t ) \betaFst
$
($s,t>0$).

A \emph{unit} of $E$ is a nonzero quasicontractive section $u$ of $E$ satisfying
\[
u_{s+t} = ( \betaEst )^* u_s \ot u_t
\qquad
(s,t>0);
\]
it is called a \emph{strong unit} if
this is strengthened to
$
\betaEst u_{s+t} =  u_s \ot u_t
$
($s,t>0$).
\end{definition}

\begin{remark}
A section $x$ of an inclusion system $E$
may be thought of as a family of bounded operators
$X = ( X_t := \dyad{x_t}{1} : \Comp_t \to E_t )_{t>0}$,
where
$( \Comp_t )_{t>0}$
is the one-dimensional inclusion system with $\Comp_t = \Comp$ ($t>0$)
and
$\beta^\Comp_{s,t}: \lambda \mapsto \lambda \ot 1 = \lambda$ ($s,t>0$).
Then $x$ is a (strong) unit if and only if $X$ is a (strong) morphism.
\end{remark}

\begin{theorem}
[\cite{BhM-inclusion}, Theorem 10]
\label{thm: 2.13}
Let
$\E$ be the product system generated by the inclusion system $E$.
Then the family of canonical maps $\iE := ( \iE_t: E_t \to \Et )_{t>0}$
forms a strong isometric morphism of inclusion systems.
Moreover
$( \iE )^* := ( ( \iE_t )^*: \Et \to E_t )_{t>0}$
restricts to a bijection from the set of units of $\E$ to the set of units of $E$,
whose inverse is denoted by $u \mapsto \uhat$.
\end{theorem}

\begin{remarks}
(i)
The quasicontractivity condition on units is crucial for the above result.

(ii)
The unit $\uhat$ of $\E$ is called the \emph{lift} of the unit $u$ of $E$;
$u = ( \iE )^* ( \uhat )$.

(iii)
For units $u$ and $v$ of $E$ and $T>0$,
\[
\ip{\uhat_T}{\vhat_T} = \lim_{{\bf t} \in J_T } \ip{u_{{\bf t}}}{v_{{\bf t}}}
\]
where,
for $n \in \Nat$ and ${\bf t} \in J_T^{(n)}$,
$u_{{\bf t}} := u_{t_1} \ot \cdots \ot u_{t_n}$.
In particular, $\uhat$ is normalised if $u$ is.

(iv)
Similarly (see~\cite{BhM-inclusion}, Theorem 11),
every morphism of inclusion systems $A:E \to F$
 lifts to a unique morphism $\Ahat: \E \to \F$ of the generated product systems.
 In terms of the corresponding canonical morphisms,
$A_t = ( \iFt )^* \Ahat_t \, \iE_t$
($t>0$).
The map $A \to \Ahat$ is a bijection between the corresponding spaces of morphisms
which respects both isometry and unitarity.
\end{remarks}

We end this section with a key example.

\begin{example} (Fock inclusion systems.)
Let $\kil$ be a separable Hilbert space.
The Fock Arveson system over $\kil$, denoted $\Fockk$,
is defined in the appendix, where the Guichardet picture of it is also described.
In the notations used there,
the \emph{Fock inclusion system over} $\kil$, denoted $F^\kil$,
is defined as follows:
\begin{align*}
F^\kil_t = \wh{\Kil_t}
&
:=
\Comp \oplus \Kil_t
\\
&
\subset
\Gamma (\Kil_t ) = \Fockk_t
\qquad
(t>0)
\end{align*}
and, in terms of the canonical identifications
\begin{align*}
&
\wh{\Kil_s} \ot \wh{\Kil_t} =
\Comp \oplus \Kil_s \oplus \Kil_t \oplus (\Kil_s \ot \Kil_t),
\ \text{ and }
\\
&
\wh{\Kil_r} \ot \wh{\Kil_s} \ot \wh{\Kil_t}
=
\Comp \oplus \Kil_r \oplus \Kil_s \oplus \Kil_t \oplus \Kil_{r,s,t}
\ \text{ where}
\\
&
\Kil_{r,s,t} :=
( \Kil_r \ot \Kil_s \op \Kil_r \ot \Kil_t \op \Kil_s \ot \Kil_t )
\op
( \Kil_r \ot \Kil_s \ot \Kil_t )
\qquad
(r,s,t>0),
\end{align*}
its structure maps are defined as follows:
for $s,t > 0$ and $(\lambda,g) \in F^\kil_{s+t}$,
\[
\betaFkst ( \lambda, g )
=
\big( \lambda, g_{[0,s[}, (S^\kil_s)^* g_{[s, s+t[}, 0 \big)
\in
\Comp \op \Kil_s \op \Kil_t \op ( \Kil_s \ot \Kil_t )
\]
or,
in the notation
$\Vac_t = (1,0) \in \wh{\Kil_t}$,
\[
\betaFkst ( \lambda, g )
=
\lambda \, \Vac_s \ot \Vac_t
+
\big( 0, g_{[0,s[} \big) \ot \Vac_t
+
\Vac_s \ot \big( 0, (S^\kil_s)^* g_{[s, s+t[} \big)
\in
\wh{\Kil_s} \ot \wh{\Kil_t}.
\]
For $r, s, t > 0$ and $(\lambda, g) \in F^\kil_{r+s+t}$,
\[
\big(
\lambda,
g_{[0,r[},
(S^\kil_r)^* g_{[r, r+s[},
(S^\kil_{r+s})^* g_{[r+s, r+s+t[},
  0
\big)
\in
\Comp \op \Kil_r \op \Kil_s \op \Kil_t \op  \Kil_{r,s,t},
\]
is a common expression for
\[
( \betaFkrs \ot I ) \betaFkrplusst ( \lambda, g )
\ \text{ and } \
( I \ot \betaFkst ) \betaFkrsplust ( \lambda, g ).
\]
In terms of the subspace inclusions
$\jmath^\kil_t: F^\kil_t \to \Fockkt$ ($t>0$),
the structure maps of the inclusion system $F^\kil$ and
Arveson system $\Fockk$ are related by
\[
\BFkst \circ \jmath^\kil_{s+t}
=
( \jmath^\kil_s \ot \jmath^\kil_t ) \circ \betaFkst
\qquad
(s,t>0).
\]
Thus
$F^\kil$ is an inclusion subsystem of $\Fockk$.

In Corollary~\ref{cor: Fk gen} below,
we verify that $F^\kil$ generates the Fock Arveson system $\Fockk$.
\end{example}

\begin{remark}
The failure
of the Fock inclusion system $F^\kil$ to be a product system
is already clearly seen through the identity
\[
\wh{\Kil_s} \ot \wh{\Kil_t} \ominus \Ran \betaFkst
=
\Kil_s \ot \Kil_t
\qquad
(s,t>0).
\]
\end{remark}

%%%%%%%%%%%%%%%%%%%%%%%%%%%%%%%%%%%%%%%%%%%%% New section

\section{Addits of pointed Arveson systems}
%\label{Product-inclusion}
\label{section: addits Arveson}

In this section we introduce the additive counterpart to the multiplicative notion of unit.
This requires the fixing of a reference unit of the Arveson system
and so is relevant to spatial Arveson systems.
We show that the space of addits then has a natural Hilbert space structure
with a one-dimensional subspace of `trivial' addits.
Elements of the orthogonal complement of this subspace are called roots and,
when the reference unit is normalised,
the subspace of roots is shown to be isomorphic to the index space of the Arveson system.
This isomorphism is established by first revealing the root space of a Fock Arveson system
with respect to the vacuum unit.

\begin{definition}
Let $(\E,u)$ be an Arveson system with (not necessarily normalised) unit.
An \emph{addit} of  $(\E,u)$
is a measurable section $a$ of $\E$ satisfying
\[
a_{s+t} = a_s \bfcdot u_t + u_s \bfcdot a_t
\qquad
(s,t>0);
\]
a \emph{root} of  $(\E,u)$ is an addit $a$ satisfying
\[
u_t \perp a_t
\qquad
(t>0).
\]
\end{definition}

\begin{remarks}
(i)
The set of addits of $(\E,u)$ forms a subspace,
denoted $\AdditEu$,
of the space of measurable sections of $\E$,
as does the set of roots,
denoted $\RootEu$.

(ii)
\emph{Normalisation}.
Let $a \in \AdditEu$ and $\lambda \in \Comp$.
Then
\[
b := ( e^{\lambda t} a_t )_{t>0} \in \AdditEv
\ \text{ for the unit } \
v := ( e^{\lambda t} u_t )_{t>0}.
\]

(iii)
\emph{Trivial addits}.
For  $\lambda \in \Comp$, $( \lambda  t u_t )_{t>0} \in \AdditEu$.
We refer to these as \emph{trivial addits} of  $(\E,u)$,
and write $\TrivEu$ for the space of these.
Note that
\[
\TrivEu \cap \RootEu = \{ 0 \},
\]
and, for $a,b \in \TrivEu$,
\[
\ip{a_t}{b_t} = t^2 \ip{a_1}{b_1} \norm{u_t}^2 / \norm{u_1}^2.
\]

(iv)
\emph{Direct sum decomposition}.
For $a \in \AdditEu$,
define
\[
a^{\Triv}:= \Big( \frac{\ip{u_t}{a_t}}{\norm{u_t}^2} \, u_t \Big)_{t>0}
\ \text{ and } \
a^{\Root}:= a - a^{\Triv}.
\]
\emph{Claim}.
$a^{\Triv} \in \TrivEu$
and
$a^{\Root} \in \RootEu$,
so
$\AdditEu = \TrivEu \oplus \RootEu$.

Since
\[
\ip{u_t}{a^{\Root}_t} = \ip{u_t}{a_t} - \ip{u_t}{a_t} = 0
\qquad
(t>0),
\]
it remains to show that $a^{\Triv}$ is a trivial addit of $(\E,u)$.
Since
\[
\frac{\ip{u_{s+t}}{a_{s+t}}}{\norm{u_{s+t}}^2}
=
\frac{\ip{u_{s}}{a_{s}}}{\norm{u_{s}}^2}
+
\frac{\ip{u_{t}}{a_{t}}}{\norm{u_{t}}^2}
\qquad
(s,t>0),
\]
the measurable function
$f_a: \Real_{>0} \to \Comp$, $t \mapsto \ip{u_t}{a_t}/\norm{u_t}^2$
satisfies Cauchy's additive functional equation
and so $f_a(t) = f_a(1) t$, in other words
$a^{\Triv} = ( \lambda t u_t )_{t>0}$
where $\lambda = \norm{u_1}^{-2} \ip{u_1}{a_1}$.
In particular $a^{\Triv} \in \TrivEu$.

(v)
Let $a, b \in \RootEu$,
and suppose that $u$ is normalised.
Then
\begin{align*}
\ip{a_{s+t}}{b_{s+t}}
&=
\ip{a_s}{b_s} \ip{u_t}{u_{t}} + \ip{u_s}{u_{s}} \ip{a_t}{b_t}
\\
&=
\ip{a_s}{b_s}   + \ip{a_t}{b_t}
\qquad
(s,t>0).
\end{align*}
Therefore, appealing to measurability once more,
\[
\ip{a_t}{b_t} = t \ip{a_1}{b_1}
\qquad
(t>0).
\]
(\emph{Cf}. (iii)).
\end{remarks}

The above remarks indicate the usefulness of the notion of pointed Arveson system
(Definition~\ref{defn: pointed}).

\begin{notation}
To a pointed Arveson system $(\E, u)$
we associate the family of bounded operators $( \thetaEut )_{t>0}$
defined by
\[
\thetaEut :=
t P_{\Comp u_1} + \sqrt{t} P_{\Comp u_1}^\perp
\in B(\E_1)
 \qquad
(t>0).
\]
\end{notation}

\begin{remarks}
Let $a, b \in \AdditEu$
for a pointed Arveson system $(\E, u)$.
For all $t>0$,
\begin{align*}
&a^{\Triv}_t
=
t a^{\Triv}_1 = t \ip{u_1}{a_1} u_1 = t P_{\Comp u_1} a_1,
\\
&\ip{u_t}{b^{\Root}_t}
=
\ip{u_t}{b_t} - \ip{u_t}{b^{\Triv}_t} = 0,
\text{ so }
a^{\Triv}_t \perp b^{\Root}_t,
\\
&\ip{a_t}{b_t}
=
\ip{a^{\Triv}_t}{b^{\Triv}_t} + \ip{a^{\Root}_t}{b^{\Root}_t}
\\
&
\quad \qquad
=
t^2 \ip{a^{\Triv}_1}{b^{\Triv}_1} + t \ip{a^{\Root}_1}{b^{\Root}_1}
=
\ip{ \theta_t a_1}{ \theta_t b_1},
\ \text{ where } \
\theta_t := \thetaEut.
\end{align*}
\end{remarks}

 \begin{proposition}
 \label{propn: AEu}
Let $(\E, u)$ be a pointed Arveson system.
Then the prescription
\begin{equation}
\label{eqn: IP on A}
\ip{a}{b} := \ip{a_1}{b_1}_{\E_1}
\end{equation}
endows the vector space $\AdditEu$ with the structure of a Hilbert space
for which the direct sum decomposition
\[
\AdditEu = \TrivEu \oplus \RootEu
\]
is an orthogonal decomposition.
 \end{proposition}

 \begin{proof}
 Set $\theta_t := \thetaEut$ ($t>0$).

 Clearly~\eqref{eqn: IP on A} defines a nonnegative sesquilinear form on $\AdditEu$.
 Suppose that $a \in \AdditEu$ satisfies $\ip{a}{a} = 0$.
 Then $a_1 = 0$ and so $\norm{a_t} = \norm{\theta_t a_1} = 0$ ($t>0$), so $a=0$.
 Thus~\eqref{eqn: IP on A} defines an inner product on $\AdditEu$.
 Suppose next that $( a^{(n)} )$ is a Cauchy sequence with respect to the induced metric on $\AdditEu$.
 Then, for all $t>0$,
 \[
 \norm{ a^{(n)}_t - a^{(m)}_t } =
 \norm{ \theta_t a^{(n)}_1 - \theta_t a^{(m)}_1 }_{\E_1}
 \les
  \norm{ \theta_t } \norm{ a^{(n)}_1 - a^{(m)}_1 }_{\E_1}
  =
  \max \{ t, \sqrt{t} \} \norm{ a^{(n)}_1 - a^{(m)}_1 }_{\E_1}
 \]
 ($n,m \in \Nat$),
 so $( a^{(n)}_t )$ is Cauchy, and thus convergent, in $\Et$.
 Set $a_t := \lim_{n \to \infty} a^{(n)}_t \in \E_t$ ($t>0$).
 Then $a$ is a measurable section of $\E$ satisfying
 \[
 a_{s+t}
 =
 \lim_{n \to \infty} a^{(n)}_{s+t}
 =
 \lim_{n \to \infty} \big( a^{(n)}_s \bfcdot u_t + u_s \bfcdot a^{(n)}_t \big)
 =
 a_s \bfcdot u_t + u_s \bfcdot a_t
 \qquad
 (s,t>0),
 \]
 so $a \in \AdditEu$.
 Moreover
 \[
 \norm{ a^{(n)} - a } =
 \norm{ a^{(n)}_1 - a_1 }_{\E_1} \to 0 \text{ as } n \to \infty.
 \]
 Therefore $\AdditEu$ is complete and thus
 a Hilbert space with respect to the inner product~\eqref{eqn: IP on A}.

 It remains to show that $( \TrivEu )^\perp = \RootEu$.
 It follows from remarks above that $\RootEu \subset ( \TrivEu )^\perp$;
 the reverse inclusion follow since
 \[
 a \in ( \TrivEu )^\perp
 \implies
 a^{\Triv}_1
 =
 \frac{\ip{u_1}{a_1}}{\norm{u_1}^2} \, u_1
 =
 0
 \implies
 a^{\Triv} = 0
 \implies
 a \in \RootEu.
 \]
\end{proof}

We next find the roots of the pointed Fock Arveson system $(\Fockk, \Vack)$,
for a separable Hilbert space $\kil$,
by working in the Guichardet picture (described in the appendix).
Thus
\[
\Vackt = \delta_\emptyset \in \Fockkt
\qquad
(t>0)
\]
and,
for $c \in \kil$ we define the measurable section
$
\chi^c:= ( c_{[0,t[} )_{t>0}
$
of $\Fockk$, in which
\[
c_{[0,t[} (\sigma) =
\left\{
 \begin{array}{ll}
c
 & \text{ if }
 \sigma \in \Gamma^{(1)}_{[0,t[}
 \\
 0
 & \text{ otherwise }
 \end{array}. \right.
\]

\begin{remark}
Both $\Vack$ and each $\chi^c$ are actually
sections of the Fock inclusion system $F^\kil$.
\end{remark}

\begin{proposition}
\label{propn: roots of Fock}
Let $\kil$ be a separable Hilbert space.
The prescription
\begin{equation}
\label{eqn: c to chi}
c \mapsto \chi^c
\qquad
(c \in \kil)
\end{equation}
defines an isometric isomorphism
from $\kil$
to $\RootFkVac$,
the space of roots of the pointed Arveson system $(\Fockk, \Vack )$.
\end{proposition}

\begin{proof}
Abbreviate
$(\Fockk, \Vack)$ to $(\Fock, \Vac)$, and $\RootFkVac$ to $\RootOmega$,
and let $\Kil_t$ be as in the appendix.

 \emph{Claim 1}.
$\chi^c \in \RootOmega$  ($c \in \kil$).

Fix $c \in \kil$.
Let $s,t>0$, then
for a.a. $\sigma$
\begin{align*}
( \chi^c_s \bfcdot \, \Vac_t &+ \Vac_s \bfcdot \, \chi^c_t )(\sigma)
\\
&=
 \chi^c_s \big( \sigma \cap [0, s[ \big)
 \,
  \delta_\emptyset \big( \sigma \cap [s, s+t[ \big)
 +
 \delta_\emptyset \big(  \sigma \cap [0, s[ \big)
 \,
 \chi^c_t \big( ( \sigma \cap [s, s+t[ ) - s \big)
 \\
 &=
 \left\{
 \begin{array}{ll}
c
 & \text{ if }
 \sigma \in \Gamma^{(1)}, \text{ and either } \sigma \subset [0,s[  \text{ or } \sigma \subset [s, s+t[
 \\
 0
 & \text{ otherwise }
 \end{array} \right.
 \\
&=
 \left\{
 \begin{array}{ll}
c
 & \text{ if }
 \sigma \in \Gamma^{(1)}_{s+t}
 \\
 0
 & \text{ otherwise }
 \end{array} \right.
 \\
& =
 \chi^c_{s+t} (\sigma),
\end{align*}
so
 $ \chi^c_{s+t} = \chi^c_s \bfcdot \, \Vac_t + \Vac_s \bfcdot \, \chi^c_t$
 ($s,t>0$).
 Since
 $\chi^c_t \perp \Vac_t$ ($t>0$),
 it follows that
 $\chi^c \in \RootOmega$.

 Now let $a \in \RootOmega$
 and set $c = V^* a_1 \in \kil$
 where $V$ is the isometry
 $\kil \to \Kil_1 \subset \Fock_1$, $c \mapsto \chi^c_1$.

 \emph{Claim 2}.
 $\esssupp a_t \subset \Gamma^{(1)}_{t}$ ($t>0$).
 \newline
 Fix $t>0$.
 For $q \in \Rat \, \cap \, ]0,t[$, $a_t = a_q \bfcdot \, \Vac_{t-q} + \Vac_q \bfcdot \, a_{t-q}$
 so,
 for a.a.    $\sigma$,
 \[
 a_t(\sigma) =
 1_{\Gamma_{[0,t[}}(\sigma)
 \big[
 a_q \big( \sigma \cap [0,q[ \big) \, \delta_\emptyset \big( (\sigma \cap [q,t[) -q \big)
 +
 \delta_\emptyset \big( \sigma \cap [0,q[ \big) \, a_{t-q}\big( (\sigma \cap [q,t[) -q \big)
 \big]
 \]
 and therefore
 \[
 a_t(\sigma) = 0 \text{ unless either } \sigma \subset [0,q[  \text{ or }  \sigma \subset [q,t[.
 \]
 Thus,
 by the countability of $\Rat$,
 there is a null set $\Nullset$ of $\Gamma_{[0,t[}$ such that
 \[
 \forall_{\sigma \in \Gamma_{[0,t[} \setminus \Nullset} \,
 \forall_{ q \in \Rat \, \cap \, ]0,t[ }\!: \, a_t(\sigma) = 0
 \text{ unless } \sigma  \subset [0,q[  \text{ or }  \sigma \subset [q,t[.
 \]
 For $\sigma = \{ s_1 < \cdots < s_n \} \in \Gamma^{(\ges 2)}_{[0,t[} \setminus \Nullset$,
 choosing $q \in \Rat$ such that
 $s_1 < q < s_2$,
 we have $\sigma \not\subset [0,q[$ and  $\sigma \not\subset [q,t[$
 so $a_t(\sigma) = 0$.
 Thus $\esssupp a_t \subset \Gamma^{(\les 1)}_{[0,t[}$.
 Since $a$ is a root of $(\Fock, \Vac)$,
 $0 = \ip{\Vac_t}{a_t} = a_t(\emptyset)$,
 thus
 $\esssupp a_t \subset \Gamma^{(1)}_{[0,t[}$.

 \emph{Claim 3}.
 $a = \chi^c$.
 \newline
 Fix $t>0$.
 By the proven Claims 2 and 1,
 $a_t, \chi^c_t \in \Kil_t \subset \Fock_t$ and $a, \chi^c \in R_\Omega$.
 It follows that,
 for each $e \in \kil$ and $s \in ]0,t[$,
 \begin{align*}
 \ip{a_t}{e_{[0,s[}}
 &=
 \ip{a_s \bfcdot \, \Vac_{t-s} + \Vac_s \bfcdot \, a_{t-s}}{\chi^e_s  \bfcdot \, \Vac_{t-s}}
 \\
 &=
 \ip{a_s}{\chi^e_s}
 \\
 &=
 s  \ip{a_1}{\chi^e_1}
 =
 s \ip{c}{e}
 =
  \ip{c_{[0,t[}}{e_{[0,s[}}
  =
   \ip{\chi^c_t}{e_{[0,s[}}.
 \end{align*}
 Therefore,
 since  $a_t, \chi^c_t \in \Kil_t$
 and the set
 $\{ e_{[0,s[}: e \in \kil, 0<s<t  \}$
 is total in $\Kil_t$,
 $a_t = \chi^c_t$.
 Thus $a = \chi^c$.

 The prescription~\eqref{eqn: c to chi} therefore defines a bijection $\kil \to \RootOmega$.
 The bijection is manifestly linear and, since
 $\norm{\chi^c}_{\RootOmega} = \norm{\chi^c_1} = \norm{ c_{[0,1[} } = \norm{c}_\kil$
 ($c\in\kil$),
 it is isometric too and thus an isometric isomorphism.
\end{proof}

\begin{corollary}
\label{cor: tensor}
Let $( \E, u ) = \big( \Fock^{\kil_1} \ot \Fock^{\kil_2}, \Vac^{\kil_1} \ot \Vac^{\kil_2} \big)$
for separable Hilbert spaces $\kil_1$ and $\kil_2$.
Then,
in the above notation,
\[
\RootEu =
\big( R_\Vac^{\Fock, \kil_1} \ot \Vac^{\kil_2} \big) \oplus \big( \Vac^{\kil_1} \ot R_\Vac^{\Fock, \kil_2} \big).
\]
\end{corollary}

\begin{proof}
Set $\kil = \kil_1 \oplus \kil_2$.

Under the natural isomorphism of pointed Arveson systems
$( \Fockk, \Vack ) \to ( \E, u )$,
the unit $\ve^e$ of $\Fockk$
maps to
the unit $\ve^{e_1} \ot \ve^{e_2}$ of $\E$,
for $e = (e_1, e_2) \in \kil$.
Therefore,
for $c = (c_1, c_2) \in \kil$,
the root $\chi^c$ of $( \Fockk, \Vack )$ maps to the root
$\chi^{c_1,c_2}$ of $( \E, u )$ given by
\begin{align*}
\chi^{c_1,c_2}_t
&=
\lim_{\lambda \to 0}
\lambda^{-1}
\big( \ve^{\lambda c_1}_t \ot \ve^{\lambda c_2}_t -
\Vac^{\kil_1} \ot \Vac^{\kil_2} \big)
\\
&=
\chi^{c_1}_t  \ot \Vac^{\kil_2}_t + \Vac^{\kil_1}_t \ot \chi^{c_2}_t
\qquad
(t>0).
\end{align*}
In view of the orthogonality relation
\[
\chi^{c_1}_t  \ot \Vac^{\kil_2}_t \perp \Vac^{\kil_1}_t \ot \chi^{c_2}_t
\qquad
( c_1 \in \kil_1, c_2 \in \kil_2, t>0 ),
\]
the result follows.
\end{proof}

Our goal now is to show that
the addits of a pointed Arveson system generate the type I part of the Arveson system.
We first show this for type I systems.

\begin{lemma}
\label{lemma: uRtoF}
Let $\kil$ be  separable Hilbert space.
Then the vacuum unit and its roots generate the Fock Arveson system $\Fockk$.
\end{lemma}

\begin{proof}
Since the set of roots of $(\Fockk, \Vack)$ is $\{ \chi^c: c \in \kil \}$,
and
$\Fockk$ is generated by its units
$\{ ( e^{\lambda t} \ve^c_t )_{t>0}: c \in \kil, \lambda \in \Comp \}$,
it suffices to prove that
\[
\big( \Vac^\kil_{2^{-n}t} + \chi^c_{2^{-n}t} \big)^{\bfcdot \, 2^n}
\to
\ve^c_t
\ \text{ as } \
n \to \infty
\qquad
(c \in \kil, t>0).
\]
Thus fix $c \in \kil$ and $t > 0$,
and set $x_n := \Vac^\kil_{2^{-n}t} + \chi^c_{2^{-n}t}$ ($n \in \Nat$).
Since
$
\| (x_n)^{ {\bfcdot} \, 2^n}\|^2 =
(1 + 2^{-n} t \norm{c}^2 )^{2^n} \les
e^{t\norm{c}^2}
=
\norm{ \ve^c_t }^2
$
($n \in \Nat$),
it suffices to prove that
\[
\big\la \ve(g), (x_n)^{{\bfcdot} \, 2^n} \big\ra
\to
\ip{\ve(g)}{\ve^c_t}
\ \text{ as } \ n \to \infty
\]
for all right continuous step functions $g \in \Kil_t$
whose discontinuities lie in the set
$\{ j 2^{-N} t: j,N \in \Nat \}$.
Thus
fix such a step function $g =\sum_{i=1}^p d^{i-1}_{[s_{i-1},s_i[}$
in which $s_0 = 0$ and $s_p = t$.
Then, for sufficiently large $n$,
\[
s_i = 2^{-n}  k_i(n) t
\ \text{ for some } \
k_i(n) \in \Nat
\qquad
(i = 1, \cdots , p).
\]
It therefore follows,
by Euler's exponential formula,
that

\begin{align*}
\ip{\ve(g)}{ (x_n)^{\bfcdot \, 2^n} }
&=
\prod_{i=1}^p
\Big\la \ve( d^{i-1}_{[0,s_i - s_{i-1}[} ) , (x_n)^{\bfcdot \, ( k_i(n) - k_{i-1}(n) )} \Big\ra
\\
&=
\prod_{i=1}^p
\big( 1 +  2^{-n} t \, \ip{ d^{i-1} }{ c } \big)^{ k_i(n) - k_{i-1}(n) }
\\
&=
\prod_{i=1}^p
\Big( 1 + \frac{ s_i - s_{i-1} }{ k_i(n) - k_{i-1}(n) } \, \ip{ d^{i-1} }{ c } \Big)^{ k_i(n) - k_{i-1}(n) }
\\
&\to
\prod_{i=1}^p
e^{ (s_i - s_{i-1}) \ip{ d^{i-1} }{ c } }
=
\ip{\ve(g)}{ \ve^c_t }
\ \text{ as } n \to \infty,
\end{align*}
as required.
\end{proof}

\begin{corollary}
\label{cor: Fk gen}
Let $\kil$ be a separable Hilbert space.
Then the product system generated by the Fock inclusion system $F^\kil$
is the Fock Arveson system $\Fockk$.
\end{corollary}

\begin{proof}
In view of the remark which precedes Proposition~\ref{propn: roots of Fock},
this follows from Lemma~\ref{lemma: uRtoF}.

\end{proof}

\begin{theorem}
\label{thm: 3.7}
Let $\E$ be a spatial Arveson system.
Let $u \in \UnitEnorm$ and
let $\F$ be the product subsystem of $\E$ generated by $u$ and all its roots.
Then the following hold.
\begin{alist}
\item
$\F = \EI$.
\item
$\RootEu = R^{\EI}_u$.
\item
$\ind \E = \dim \RootEu$.
\end{alist}
\end{theorem}

\begin{proof}
Let $\kil$ and $\hil$ be the Hilbert spaces $\kil(\E)$ and $\RootEu$
respectively,
and let $F$ be the inclusion subsystem of $\E$ generated by $u$ and all of its roots.
Thus $\dim \kil = \ind \E$,
$\Fockk \cong \E^I$
and $\F$ is the product subsystem of $\E$ generated by $F$.
Recall that, by Theorem~\ref{thm: always}, $\F$ is an Arveson subsystem of $\E$.

(a)
We first show that $\F$ is a product subsystem of $\EI$.
By Proposition~\ref{propn: roots of Fock},
the following prescription defines
unitary operators
\[
A_t:
F^\hil_t \to F_t,
\quad
\lambda \Vac^\hil_t + \chi^a_t \mapsto \lambda u_t + a_t
\qquad
( \lambda \in \Comp, a \in \hil = \RootEu, t>0),
\]
and
it is easily seen that $A = ( A_t )_{t>0}$
is an isomorphism of inclusion systems.
By Corollary~\ref{cor: Fk gen},
the product system generated by $F^\hil$ is $\F^\hil$.
By
Remark (iv) after Theorem~\ref{thm: 2.13},
$A$ lifts to an isomorphism of product systems $\wh{A}: \Fock^\hil \to \F$.
Theorem~\ref{thm: 7.16},
together with the remark following it,
imply that $\wh{A}$ is an isomorphism of Arveson systems,
and so $\F$ is of type I.
It follows
that $\F$ is a product subsystem of $\EI$.

We next show that $\EI$ is a product subsystem of $\F$,
equivalently that,
for any normalised unit $v$ of $\E$,
$v_t \in \F_t$ ($t>0$).
To this end,
let $v \in \UnitEnorm$ and fix an isomorphism of pointed Arveson systems
$\psi: ( \Fockk, \Vack ) \to ( \EI, u )$.
Then
$( e^{ixt} v_t )_{t>0} = \psi (\vp^c )$
for some $c \in \kil$ and $x \in \Real$.
Set
$a := \psi (\chi^c) \in R^{\EI}_u$.
Since any root of $(\EI, u)$ is
 a root of $(\E, u)$,
$a_s \in F_s$ ($s>0$) and
\[
\psi_t( \ve^c_t ) =
\lim_{n \to \infty}
\psi_t \big( \Vac^\kil_{2^{-n} t} + \chi^c_{2^{-n} t}\big)^{ \bfcdot \, 2^n } =
\lim_{n \to \infty}
 \big( u_{2^{-n} t} + a_{2^{-n} t}\big)^{ \bfcdot \, 2^n }
 \in \Ft
\]
so
$
v_t =
e^{-ixt} e^{-t \norm{c}^2/2 } \, \psi_t (\ve^c_t ) \in \F_t
$
($t>0$),
as required.

Therefore $\F = \EI$, so (a) holds.
(b) follows from (a).

(c)
By (a) we have isomorphisms of Arveson systems
\[
\Fockh \cong \F = \EI \cong \Fockk.
\]
This implies that $\hil \cong \kil$, and so
$
\ind \E = \dim \kil = \dim \hil = \dim \RootEu
$.
\end{proof}

%%%%%%%%%%%%%%%%%%%%%%%%%%%%%%%%%%%%%%%%%%%%% New section

\section{Addits of pointed inclusion systems}
%\label{Product-inclusion}
\label{section: addits inclusion}

In this section we extend notions of the previous section to inclusion systems,
and show that, as with units,
addits of an inclusion system lift to addits of the generated product system.

We call an ordered pair $(E,u)$,
consisting of an inclusion system $E$ and a normalised unit $u$ of $E$,
a \emph{pointed inclusion system}.

\begin{definition}
(\cite{BhM-inclusion})
Let $(E,u)$ be a pointed inclusion system.
An \emph{addit of}  $(E,u)$ is a section $a$ of $E$
satisfying the additivity condition
\[
a_{s+t} = ( \betaEst )^* ( a_s \ot u_t + u_s \ot a_t )
\qquad
(s,t>0),
\]
and the following boundedness condition: there is $k \in \Rplus$ such that
\[
\norm{a_t}^2 \les k ( t + t^2 )
\qquad
(t>0).
\]
An addit $a$ of $(E,u)$ is a \emph{root} if it satisfies
\[
a_t \perp u_t
\qquad
(t>0).
\]
\end{definition}

We now establish the additive counterpart to
(the second part of)
Theorem~\ref{thm: 2.13},
whose notations we continue to adopt.

\begin{proposition}
\label{propn: 3.12}
Let $(E,u)$ be a pointed inclusion system.
Let $\E$ be the product system generated by $E$,
and let $\uhat$ be the lift of $u$.
Then the following hold.
\begin{alist}
\item
The map
$( \iE )^* : \big( ( \iE_t )^*: \Et \to E_t \big)_{t>0}$
restricts to a bijection from the set of addits of $(\E, \uhat)$
to the set of addits of $(E,u)$,
whose inverse is denoted by $a \mapsto \ahat$.
\item
If $a$ is a root of $(E,u)$ then $\ahat$ is a root of $(\E, \uhat)$.
\end{alist}
\end{proposition}

\begin{proof}
Let us drop the superscripts on
$\betaE$, $\BE$ and $\iE$.

(a)
First let $b$ be an addit of $(\E, \uhat)$.
Then $\imath^*(b)$ is an addit of $(E,u)$ since
\begin{align*}
\betast^* \big(
\is^* b_s \ot u_t + u_s \ot \imt^* b_t
\big)
&=
\big( (\is \ot \imt ) \circ \betast \big)^*
\big( b_s \ot \uhat_t + \uhat_s \ot b_t \big)
\\
&=
\big( \Bst \circ \isplust \big)^*
\big( b_s \ot \uhat_t + \uhat_s \ot b_t \big)
=
\isplust^* b_{s+t}
\qquad
(s,t>0).
\end{align*}
Let
$\alpha$ denote the resulting map
from addits of  $(\E, \uhat)$
to addits of $(E,u)$.
Suppose that addits $b^1$ and $b^2$ of $(\E, \uhat)$
satisfy
$\alpha (b^1) = \alpha (b^2)$.
Fix $T > 0$.
An induction on $n$
confirms that,
for any addit $b$ of $(\E, \uhat)$,
\begin{align*}
&
\BTbft \, b_T =
\sum_{j=1}^n
\uhat_{t_1} \ot \cdots \ot \uhat_{t_{j-1}}
\ot b_{t_j} \ot
 \uhat_{t_{j+1}} \ot \cdots \ot \uhat_{t_n},
 \ \text{ and}
 \\
&
\BTbft \circ \ibft =
\imath_{t_1} \ot \cdots \ot \imath_{t_n}
\qquad
( n \in \Nat, {\bf t} \in J^{(n)}_T ).
\end{align*}
Therefore,
for any addit $b$ of $(\E, \uhat)$,
\begin{align*}
\ibft^* b_T
&=
\big( ( \imath_{t_1}^* \ot \cdots \ot \imath_{t_n}^* ) \circ  \BTbft \big) b_T
\\
&=
\sum_{j=1}^n
u_{t_1} \ot \cdots \ot u_{t_{j-1}}
\ot \imath_{t_j}^* b_{t_j} \ot
 u_{t_{j+1}} \ot \cdots \ot u_{t_n}
\qquad
( n \in \Nat, {\bf t} \in J^{(n)}_T ).
\end{align*}
Now
the RHS is the same for $b=b^1$ and $b=b^2$,
therefore
$\ibft^* b^1_T = \ibft^* b^2_T$
(${\bf t} \in J_T$).
Since the net
$\big( \ibft \ibft^* \big)_{{\bf t} \in J_T}$
converges strongly to $\IET$,
it follows that $b^1_T = b^2_T$.
Unfixing $T$ we conclude that $b^1 = b^2$,
and so $\alpha$ is injective.

Since the trivial addits of $(\E, \uhat)$
are clearly mapped by $\alpha$ onto the trivial addits of $(E,u)$,
in order to establish the surjectivity of $\alpha$
it suffices to fix a root $a$ of  $(E,u)$ and
find a root, $\ahat$ say, of $(\E, \uhat)$ such that
$\imath^*( \ahat ) = a$.
Accordingly,
let $a$ be a root of $(E,u)$, with boundedness constant $k$ and fix $T>0$.

\emph{Claim 1}.
Setting
$\abft :=
\sum_{j=1}^n
u_{t_1} \ot \cdots \ot u_{t_{j-1}}
\ot a_{t_j} \ot
 u_{t_{j+1}} \ot \cdots \ot u_{t_n}$
($ n \in \Nat, {\bf t} \in J^{(n)}_T $),
the net
$\big( \ibft \abft \big)_{{\bf t} \in J_T}$
converges.

First note that the net is bounded
since
$a_t \perp u_t$ ($t>0$),
so
\[
\norm{ \ibft \abft }^2 =
\norm{ \abft }^2 =
\sum_{j=1}^n \norm{ a_{t_j} }^2 \les
k \sum_{j=1}^n ( t_j + t_j^2 ) \les k ( T + T^2 )
\qquad
( n \in \Nat, {\bf t} \in J^{(n)}_T ).
\]
Next note the identity
\[
\ibfs^* \ibft \abft = \betabfst^* \abft = \abfs
\qquad
( {\bf s} \les {\bf t} \text{ in } J_T ).
\]
Fix
$x \in \ET$ and $\ve > 0$.
Choose
${\bf r} \in  J_T$
such that
$\norm{ x - \ibfr \ibfr^* x } < \ve$.
Then,
for ${\bf t} \ges {\bf r}$,
\[
| \ip{ \ibft \abft - \ibfr \abfr}{x} |^2
=
| \ip{ \ibft \abft }{ ( I - \ibfr \ibfr^* ) x } |^2
\les
k ( T + T^2 ) \ve^2.
\]
It follows that
$\big( \ibft \abft \big)_{{\bf t} \in J_T}$
is weakly Cauchy.
Set
$\ahat_T := \weaklim_{{\bf t} \in J_T} \ibft \abft$.
Now
\[
\ibfs \ibfs^* \ibft \abft = \ibfs \abfs
\qquad
({\bf s} \les {\bf t} \text{ in } J_T),
\]
therefore
$\ibfs \ibfs^* \ahat_T = \ibfs \abfs$ (${\bf s} \in J_T$).
It follows that
$\ibfs \abfs \to \ahat_T$ (in norm), as claimed.

\emph{Claim 2}.
Setting
$\ahat := ( \ahat_T )_{T>0}$,
$\ahat$
is an addit of $(\E, \uhat )$ such that $\imath^* ( \ahat ) = a$.

Let $S,T>0$.
Write $\ubft$ for
$\sum_{j=1}^n u_{t_1} \ot \cdots \ot u_{t_n}$
($ n \in \Nat, {\bf t} \in J^{(n)}_T $).
Then,
for ${\bf s} \in  J_S$ and ${\bf t} \in  J_T$,
\begin{align*}
( \ibfs \ot \ibft ) ( \abfs \ot \ubft + \ubfs \ot \abft )
&=
( \ibfs \ot \ibft ) a_{ {\bf s} \smile {\bf t} }
=
\BST \, \imath_{ {\bf s} \smile {\bf t} } a_{ {\bf s} \smile {\bf t} }.
\end{align*}
Taking limits and using the fact that the net $( \ibfr \ubfr )_{ {\bf r} \in  J_R }$
converges to $\uhat_R$ ($R>0$),
we see that
\[
\ahat_S \ot \uhat_T + \uhat_S \ot \ahat_T = \BST \ahat_{S+T}.
\]
Thus
$\ahat$ is an addit of $\uhat$.
Now,
since
\[
\imath_T^* \ibft \abft = \betaTbft^* \abft = a_T
\qquad
( {\bf t} \in  J_T ),
\]
it follows that $\imath_T^* \ahat_T = a_T$,
and Claim 2 is established.

Therefore $\alpha$ is also surjective and so (a) follows.

(b)
Let $a$ be a root of $(E,u)$.
Then
\[
\ip{ \ibft \abft  }{ \ibft  \ubft  }
=
\ip{ \abft  }{ \ubft  }
=
0
\qquad
( T>0, {\bf t} \in  J_T ).
\]
Taking limits we see that
$\ip{ \ahat_T }{ \uhat_T } = 0$ ($ T>0 $),
so $\ahat$ is a root of $\uhat$.
The proof is now complete.
\end{proof}

%%%%%%%%%%%%%%%%%%%%%%%%%%%%%%%%%%%%%%%%%%%%% New section

\section{Amalgamation}
%\label{Independent}
\label{section: amalgamated products}

The amalgamation of Arveson systems, via a contractive morphism,
was introduced in~\cite{BhM-inclusion}.
This generalised a construction of Skeide
which corresponds to the case where
the morphism is given by Dirac dyads from normalised units (\cite{Ske-index}).
A formula for its index,
in terms of that of the constituent systems, was given in~\cite{Muk-index}.
In this section we first show how the root space
of an amalgamated product of pointed Arveson systems
(defined to be that given by the corresponding morphism of Dirac dyads)
 may be expressed
in terms of the root spaces of its constituent systems,
when the morphism is partially isometric.
The amalgamated product of pointed Arveson systems
 may be realised as a product subsystem of the tensor product Arveson system
(\cite{Muk-index}, Theorem 2.7);
we give an explicit formula for the subsystem which shows,
in particular, that it is independent of the fixed normalised units
and so depends only on the underlying Arveson systems.
The latter fact may alternatively be proved using random sets
(\cite{Lie-spatial}), or directly (\cite{BLMS-intrinsic});
see also~\cite{BLS-Powers}.
The section ends with a new formula for the space of roots of
the tensor product of two pointed Arveson systems.

To begin we quote a basic result.

 \begin{theorem}
 [\cite{BhM-inclusion}, Section 3;~\cite{Muk-index}, Theorem 2.7]
 %\label{universal}
 \label{thm: 2.14}
 Let $C: \Etwo \to \Eone$ be a contractive morphism between Arveson systems.
 Then there is a triple $(\E, J^1, J^2)$,
 unique up to isomorphism,
 consisting of a product system $\E$
 and isometric morphisms of product systems $J^i: \E^i \to \E$ $(i=1,2)$
 such that
 \begin{rlist}
 \item
 $( J^1_t )^* J^2_t = C_t$ $(t>0)$, and
 \item
 $\E = J^1 ( \Eone ) \bigvee J^2 ( \Etwo )$.
 \end{rlist}
 Notation\tu{:} $\Eone \ot_C \Etwo$.
Terminology\tu{:} the \emph{amalgamated product of $\Eone$ and $\Etwo$ via $C$}.

Conversely,
let $\Eone$ and $\Etwo$ be product subsystems of an Arveson system $\F$
with
inclusion morphisms $J^i: \E^i \to \F$ $(i=1,2)$.
Then
$\Eone \bigvee \Etwo = \Eone \ot_C \Etwo$
where
$C = \big( ( J^1_t )^* J^2_t \big)_{t>0}$.
 \end{theorem}

\begin{remarks}
The construction of $\Eone \ot_C \Etwo$ is via an inclusion system.
In case
$\Eone$ and $\Etwo$ are product subsystems of an Arveson system $\F$,
$\Eone \ot_C \Etwo$ may be realised as
the product subsystem  of $\F$
generated by
the inclusion subsystem $( \Eonet \vee \Etwot )_{t>0}$,
in particular it is an Arveson system,
by Theorem~\ref{thm: always}).

When $C$ takes the form
$\big( \dyad{u^1_t }{ u^2_t } \big)_{t>0}$
for normalised units $u^i$ of $\E^i$ ($i=1,2$),
the case treated in~\cite{Ske-index},
$\Eone \ot_C \Etwo$ is denoted $\Eone \ot_{u^1,u^2} \Etwo$.
\end{remarks}

The following proposition is
a straight-forward consequence of Theorem~\ref{thm: 2.14}.

\begin{proposition}
\label{propn: 4.1}
Let $(\E,u)$ and $(\F,v)$ be pointed Arveson systems.
Then
\[
\E \ot_{u,v} \F \cong ( \E \ot v ) \bigvee ( u \ot \F ).
\]
\end{proposition}

\begin{notation}
For a pointed Arveson system $(\E, u)$,
we set
\[
\RootEucal :=
\{ a_1: a \in \RootEu \}.
\]
Thus
$\RootEucal$ is a closed subspace of the Hilbert space $\E_1$;
by definition,
$\RootEu \cong \RootEucal$.
\end{notation}

\begin{theorem}
\label{thm: 3.13}
Let $\E = \Eone \ot_C \Etwo$
for spatial Arveson systems $\Eone$ and $\Etwo$
and a partially isometric morphism $C: \Etwo \to \Eone$,
and let $u^2 \in \UnitEtwonorm$.
Suppose that
$\E$ is an Arveson system
and that
$u^2$ lies in the initial space of $C$\tu{:}
$C^*_t C_t u^2_t = u^2_t$ $(t>0)$.
Then $u^1 := C u^2$ is a unit which is identified with $u^2$ in $\E$ and,
denoting the common unit in $\E$ by $u$,
\[
\RootEucal = \RootEoneuonecal \op_{C_1} \RootEtwoutwocal.
\]
\end{theorem}

\begin{proof}
It follows from Theorem~\ref{thm: 2.14}
that we may identify $\Eone$ and $\Etwo$ with
subsystems of $\E$,
and $C$ with $( ( J^1_t )^* J^2_t )_{t>0}$
where $J^1$ and $J^2$ are the corresponding inclusion morphisms.
By Proposition 2.10 of~\cite{Muk-index},
the projections $P_{\Eonet}$ and $P_{\Etwot}$ commute,
so $P_{\Eonet \cap \Etwot} = P_{\Eonet} P_{\Etwot}$ ($t>0$).
Thus $\Eone \cap \Etwo := \big( \Eonet \cap \Etwot \big)_{t>0}$
is a product subsystem of $\E$.
Under this identification
$u^2$ and $u^1$ are identified, and
$\RootEoneuonecal \op_{C_1} \RootEtwoutwocal$ coincides with
$\RootEoneuonecal \vee \RootEtwoutwocal$ in $\RootEucal$.
The theorem is therefore proved once it is shown that
$\RootEoneuonecal \vee \RootEtwoutwocal = \RootEucal$.

Let $a \in \RootEu$ and set
$c := \big( J^1_t b^1_t +  J^2_t b^2_t -  J_t b_t \big)_{t>0}$
where
$b^1_t = (J^1_t)^* a_t$,
$b^2_t = (J^2_t)^* a_t$
and
$b_t = (J_t)^* a_t$
($t>0$),
and
$J$ denotes the inclusion morphism $\Eone \cap \Etwo \to \E$.
Thus
\[
c_t = \big( P_{\Eonet} + P_{\Etwot} - P_{\Eonet \cap \Etwot} \big) a_t
=
P_{ \Eonet \vee \Etwot } \, a_t
\qquad
(t>0).
\]
\emph{Claim}.
$c \in \RootEu$.
First note that
\begin{align*}
b_s^1 \ot u_t^1 + u_s^1 \ot b_t^1
&=
( J^1_s \ot J^1_t )^* ( a_s \ot u_t + u_s \ot a_t )
\\
&=
( J^1_s \ot J^1_t )^* \BEst a_{s+t}
=
\BEonest ( J^1_{s+t} )^* a_{s+t} = \BEonest b^1_{s+t}
\qquad
(s,t>0)
\end{align*}
so $b^1 \in \RootEoneuone$.
Similarly,
$b^2 \in \RootEtwoutwo$ and $b \in R_u^{\Eone \cap \Etwo}$.
Thus
$J^1 b^1, J^2 b^2, J b \in \RootEu$, so  $c \in \RootEu$.

 Now $E := \big( \Eonet \vee \Etwot \big)_{t>0}$
 is an inclusion subsystem which generates the Arveson system $\E$,
 and
 \[
 0 =
 P_{ \Eonet \vee \Etwot } ( a_t - c_t )
 =
 P^E_t ( a_t - c_t )
 =
 \iEt ( \iEt )^* ( a_t - c_t )
 \qquad
 (t>0),
 \]
so
$( \iE )^* ( a - c ) = 0$.
Since
$(a-c) \in \RootEu$,
it follows from
Proposition~\ref{propn: 3.12} that
$a-c = 0$.
Thus
\[
a_1 = c_1 = J^1_1 b_1^1 + J^2_1 b^2_1 - J_1 b_1
\in
\RootEoneucal + \RootEtwoucal
\subset
\RootEoneucal \vee \RootEtwoucal.
\]
The result follows.
\end{proof}

\begin{corollary}
Let $(\Eone, u^1)$ and $(\Etwo, u^2)$ be pointed Arveson systems.
Then, identifying $\E = \Eone \ot_{u^1,u^2} \Etwo$
with
$( \Eone \ot u^2 ) \bigvee ( u^1 \ot \Etwo )$,
and letting $u$ denote $u^2$ identified with $u^1$,
\[
\RootEucal = \RootEoneuonecal \op_{C_1} \RootEtwoutwocal,
\]
for the partially isometric morphism
$C:= \big( \dyad{ u^1_t }{ u^2_t } \big)_{t>0} : \E^2 \to \E^1$.
\end{corollary}

Note that, in this case,
we directly see the orthogonality
\[
\ip{ a^1_1 }{ a^2_1 }_{C_1}
=
\ip{ a^1_1 }{ C_1 a^2_1 }
=
\ip{ a^1_1 }{ u^1_1 } \, \ip{ u^2_1 }{ a^2_1 }
=
0
\qquad
(a^1 \in \RootEoneuone,
a^2 \in \RootEtwoutwo).
\]

\begin{remark}
Root spaces need not behave well under amalgamation
over contractive morphisms that are not partially isometric.
\end{remark}

\begin{example}
Fix $\lambda \neq 0$.
Set $\E = \Eone \ot_C \Etwo$
where $\Eone = \Etwo = \Fockk$ for
the trivial Hilbert space $\kil = {\{0\}}$,
and
$C = \big( \dyad{ u^1_t }{ u^2_t } \big)_{t>0}$
for the units
$u^1 := \Vac^{\{0\}}$
and
$u^2 := \big( e^{ - t \lambda^2/2 } \Vac_t^{\{0\}} \big)_{t>0}$.
Theorem 2.7 of~\cite{Muk-index}
implies that
$\E$ is isomorphic to the product system generated by the
normalised units
$\Vac^\Comp$ and $\vp^\lambda$ of $\Fock^\Comp$,
in other words
$\E$ is isomorphic to the Fock Arveson system $\Fock^\Comp$ itself.
Thus
\[
\RootEoneuonecal = \{ 0 \} = \RootEtwoutwocal,
\ \text{ but, for any unit $u$ of $\E$, } \
\RootEucal \cong \Comp.
\]
\end{example}

For an inclusion subsystem $F$ of an Arveson system $\E$,
consider the following family of
orthogonal projections in $B( \E_1 )$:
\begin{equation}
\label{eqn: PF}
P^F_{r,t} :=
\left\{
 \begin{array}{ll}
 P_{F_{t}} \ot \IEoneminust
&
\text{ if }
0=r<t<1
\\
P_{F_{1}}
&
\text{ if }
0=r \text{ and } t=1
\\
\IEr \ot P_{F_{t-r}} \ot \IEoneminust
&
\text{ if }
0<r<t<1
\\
\IEr \ot P_{F_{1-r}}
&
\text{ if }
0<r<t=1.
\end{array}, \right.
\end{equation}
It follows from Theorem~\ref{thm: Lie random} below,
and the first remark following it,
that,
for a product subsystem $\F$ of $\E$,
\[
\PFst \to \IEone
\ \text{ as } \
(t-s) \to 0.
\]

\begin{theorem}
\label{thm: 4.2}
Let $\E$ and $\F$ be spatial Arveson systems.
Then,
for any normalised units $u$ and $v$ of $\E$ and $\F$ respectively,
\[
\E \ot_{u,v} \F
\cong
(\E \ot \FI ) \bigvee ( \EI \ot \F ).
\]
\end{theorem}

\begin{proof}
Let $u \in \UnitEnorm$ and $v \in \UnitFnorm$.
Set $\G := (\E \ot v ) \bigvee ( u \ot \F )$ and,
for $n \in \Nat$ and $i \in \{ 1, \cdots , n \}$,
set
$
P^n_i := P^{\Comp u}_{s,t}
$
where
$\Comp u$ denotes the product subsystem of $\E$ generated by $u$,
and
$(s,t) = ( (i-1)/n, i/n )$.
By Proposition~\ref{propn: 4.1},
$\G \cong \E \ot_{u,v} \F$,
and so, by symmetry,
it suffices to show that
$\E \ot \FI$ is a product subsystem of $\G$.
By Theorem~\ref{thm: 3.7}
it suffices to show that $z \ot a_t \in \Gt$
for $t>0$, $z \in \Et$ and $a \in \RootFv$.
The argument we give, for the case $t=1$
easily adjusts to deal with general $t>0$.
Thus
let $z \in \E_1$ and $a \in \RootFv$ with $\norm{ a } = 1$.

Let $\ve > 0$.
Choose $n \in \Nat$
such that
$\norm{ z - P^{ \Comp u }_{s,t} z } \les \ve$
for
$(t-s) \les 1/n$
and take the root decomposition
\[
a_1
=
\sum_{i=1}^n x^i
\ \text{ where } \
x^i = ( v_{1/n} )^{ \bfcdot \, (i-1) } \bfcdot \, a_{1/n} \bfcdot \, ( v_{1/n} )^{ \bfcdot \, (n-i) }
\qquad
( i = 1, \cdots , n ).
\]
Thus
$\norm { x^i } = \norm { a_{1/n} } = 1/\sqrt{n}$
for each $i$ and,
since $x^i \perp x^j$ for $i \neq j$,
\[
\big\|
z \ot a_1 - \sum\nolimits_{i=1}^n P^n_i z \ot x^i
\big\|^2
=
\big\|
\sum\nolimits_{i=1}^n ( z - P^n_i z ) \ot x^i
\big\|^2
=
\frac{1}{n}
\sum_{i=1}^n \norm{  z - P^n_i z }^2
\les
\ve^2.
\]
We must therefore show that
$P^n_i z \ot x^i \in \Gone$ ($n \in \Nat, i=1, \cdots , n$).
Accordingly,
fix $n \in \Nat$ and $i \in \{ 1, \cdots , n \}$.
Note that
\[
P^n_i z
\in\Linbar
\big\{
c^1 \bfcdot \, \cdots \bfcdot \, c^{i-1} \bfcdot \, u_{1/n} \bfcdot \, c^{i+1} \bfcdot \, \cdots \bfcdot \, c^n:
c^1, \cdots , c^n \in \Eoneovern
\big\}
\]
and,
for $c^1, \cdots , c^n \in \E_{1/n}$,
\begin{align*}
&
\big(
c^1 \bfcdot \, \cdots \bfcdot \, c^{i-1} \bfcdot \, u_{1/n} \bfcdot \, c^{i+1} \bfcdot \, \cdots \bfcdot \, c^n
\big) \ot x^i
=
\\
&
\qquad \quad
( c^1 \ot v_{1/n} ) \bfcdot \, \cdots \bfcdot \, ( c^{i-1}  \ot v_{1/n} ) \bfcdot \, ( u_{1/n}  \ot a_{1/n} ) \bfcdot \,
( c^{i+1} \ot v_{1/n} ) \bfcdot \, \cdots \bfcdot \, ( c^n \ot v_{1/n} ),
\end{align*}
whilst
$$
c^j \ot v_{1/n} \in \Eoneovern \ot v_{1/n} \subset \Goneovern \ (j \neq i)
\ \text{ and } \
u_{1/n}  \ot a_{1/n} \in u_{1/n}  \ot \Foneovern \subset \Goneovern.
$$
It follows that
$
P^n_i z \ot x^i \in \big( \Goneovern \big)^{ \bfcdot \, n } \subset \Gone
$,
as required.
\end{proof}

\begin{remark}
This result reaffirms justification for referring to the above
(spatial) Arveson system as
the \emph{spatial product} of the spatial Arveson systems $\E$ and $\F$.
\end{remark}

\begin{corollary}
\label{cor: 4.3}
Let $\E$ and $\F$ be spatial Arveson systems.
Then,
for normalised units $u$ and $v$ of $\E$ and $\F$ respectively,
\[
(\EI \ot v ) \bigvee ( u \ot \FI )
=
\EI \ot \FI
=
( \E \ot \F )^I.
\]
\end{corollary}

\begin{proof}
The first identity follows from
Proposition~\ref{propn: 4.1}
and
Theorem~\ref{thm: 4.2}.
The second is well-known;
it is a consequence of the following identity
(see~\cite{Arv-noncommutative}, Corollary 3.7.3):
\begin{equation}
\label{eqn: UEF}
\UnitEtensorF
=
\big\{ u \ot v: u \in \UnitE, v \in \UnitF \big\}.
\end{equation}
\end{proof}

Our next result is the counterpart for roots of the identity~\eqref{eqn: UEF} for units.
It generalises Corollary~\ref{cor: tensor}.

\begin{theorem}
\label{thm: 4.5}

Let $( \E, u )$ and $(\F, v)$ be pointed Arveson systems.
Then
\[
\RootEtensorFutensorv
=
( \RootEu \ot v ) \op ( u \ot \RootFv ).
\]
\end{theorem}

\begin{proof}
First note that,
by Theorem~\ref{thm: 3.7}
and the identity
$
( \E \ot \F )^I = \EI \ot \FI
$,
we may suppose without loss that
$\E$ and $\F$ are type I Arveson systems.
Writing
$(\Eone, u^1)$ and $(\Etwo, u^2)$
for $( \E, u )$ and $(\F, v)$ respectively,
and setting
$\kil^i := \kil( \Ei )$ ($i=1,2$),
there are
isomorphisms of pointed Arveson systems
$\phi^i: ( \Fock^{\kil_i}, \Vac^{\kil_i} ) \to ( \Ei, u^i )$
($i=1,2$).
Since the isomorphism
$
\phi^1 \ot \phi^2:
\big(
\Fock^{\kil_1} \ot \Fock^{\kil_2}, \Vac^{\kil_1} \ot \Vac^{\kil_2}
\big)
\to
\big(
\Eone \ot \Etwo, u^1 \ot u^2
\big)
$
restricts to a bijection of roots,
and maps
\[
R^{ \Fock, \kil_1 }_\Vac \ot \Vac^{ \kil_2 }
\text{ to }
R^{\Eone}_{u^1} \ot u^2
\ \text{ and } \
\Vac^{ \kil_1 } \ot R^{ \Fock, \kil_2 }_\Vac
\text{ to }
u^1 \ot R^{\Eone}_{u^2},
\]
the result follows from
Corollary~\ref{cor: tensor}.
\end{proof}

%%%%%%%%%%%%%%%%%%%%%%%%%%%%%%%%%%%%%%%%%%%%% New section

\section{Cluster construction}
%\label{Cluster}
\label{section: cluster construction}

In the first half of this section we develop a \emph{cluster construction}
for product subsystems of an Arveson system,
and show how the construction leads to a new description of
the type I part of a spatial Arveson system. In the second half
we relate our construction to the Cantor--Bendixson derivative
which sends a closed subset of the unit interval to its `cluster',
namely the collection of its accumulation points,
via the connection to random sets elaborated in~\cite{Lie-random}.

\begin{notation}
For an inclusion subsystem $F$ of an Arveson system $\E$, and $t>0$,
set
\[
\Fminusperpt := \Et \ominus \Fminust
\ \text{ where } \
\Fminust := \underset{0<r<t}{\vee} \,
( \Er \ominus F_r ) \ot ( \Etminusr \ominus F_{t-r} ).
\]
\end{notation}

\begin{proposition}
\label{propn: 7.1}
Let $F$ be an inclusion subsystem of an Arveson system $\E$.
Then $\Fminusperp := \big( \Fminusperpt \big)_{t>0}$
is an inclusion subsystem of $\E$ containing $F$.
\end{proposition}

The proof of this proposition is no easier than that of
its generalisation, Proposition~\ref{propn: 5.6},
which is given there (and does not depend on any of the intervening theory).

\begin{definition}
\label{defn: 6.2}
Let $\F$ be a product subsystem of an Arveson system $\E$.
The \emph{cluster of $\F$ in $\E$} is
the product system generated by the inclusion system $\FFminusperp$.
We denote it $\F\,\check{}$.
\end{definition}

\begin{lemma}
\label{lemma: 5.3}
Let $\F$ be  a product subsystem of an Arveson system $\E$,
and let
$s,t>0$.
Then the following hold.
\begin{alist}
\item
$ \FFminusperps \ot \Fockt \subset \FFminusperpsplust$
and
$\Focks \ot \FFminusperpt \subset \FFminusperpsplust$.
\item
$( \FFminusperps \ominus \Focks ) \ot \Fockt \subset \FFminusperpsplust \ominus \Fsplust$
and
$\Fs \ot ( \FFminusperpt \ominus \Fockt ) \subset \FFminusperpsplust \ominus \Fsplust$.
\end{alist}
\end{lemma}

\begin{proof}
Let $s,t>0$.
(a)
Let $r>0$ satisfy $0<r<s+t$.

If
$r<s$, then
\begin{align*}
\FFperpr \ot \FFperpsplustminusr
&=
\FFperpr \ot ( \Focksminusr \ot \Fockt )^\perp
\\
&=
\FFperpr \ot
\big(
\FFperpsminusr \ot \Fockt \, \op \, \Esminusr \ot \FFperpt
\big)
\subset
\FFminuss \ot \Fockt \, \op \, \Es \ot \FFperpt.
\end{align*}

If
$r=s$, then
\[
\FFperpr \ot \FFperpsplustminusr
=
\FFperps \ot \FFperpt
\subset
\Es \ot \FFperpt.
\]

If
$r>s$, then
\begin{align*}
\FFperpr \ot \FFperpsplustminusr
&\subset
\Es \ot \Erminuss \ot \FFperpsplustminusr
\subset
\Es \ot ( \Fockrminuss \ot \Focksplustminusr )^\perp
=
\Es \ot \FFperpt.
\end{align*}
Therefore
\[
\FFminussplust
\subset
\FFminuss \ot \Fockt \op \Es \ot \FFperpt
=
( \FFminusperps \ot \Fockt )^\perp.
\]
The first inclusion follows.
The second now follows by symmetry.

(b)
Since $\F$ is a product subsystem of $\E$,
the first inclusion in
(b) follows from the first inclusion in (a):
\begin{equation*}
( \FFminusperps \ominus \Focks ) \ot \Fockt
=
\FFminusperps \ot \Fockt \, \ominus \, \Focks \ot \Fockt
\subset
\FFminusperpsplust \ominus \Focksplust.
\end{equation*}
The second inclusion in (b) follows similarly.
\end{proof}

\begin{corollary}
\label{cor: 5.3plus}
Let $(\E,u)$ be a pointed Arveson system,
and set $\Fock = \Comp u$.
Then, for $s,t>0$,
\begin{align*}
&
 \FFminusperps \ot u_t \subset \FFminusperpsplust
 \ \text{ and } \
( \FFminusperps \ominus \Focks ) \ot u_t \subset
\FFminusperpsplust \ominus \Focksplust;
\\
&
u_s \ot \FFminusperpt \subset \FFminusperpsplust
\ \text{ and } \
u_s \ot ( \FFminusperpt \ominus \Fockt ) \subset
\FFminusperpsplust \ominus \Focksplust.
\end{align*}
\end{corollary}

\begin{notation}
For a pointed Arveson system $(\E, u)$,
set
\[
\XEut := ( \Comp u_t )^{\ominus \perp} \ominus \Comp u_t
\qquad
(t>0),
\]
and define isometries
\[
\jEust: \XEus \to \XEut,
\quad
x \mapsto x \bfcdot \, u_{t-s}
\qquad
(0<s<t).
\]
Then $\big( (\XEut )_{t>0}, ( \jEurs )_{0<r<s} \big)$
is easily seen to form an inductive system of Hilbert spaces.
Let
$\big( \XEu, ( \jEut : \XEut \to \XEu )_{t>0} \big)$
denote its inductive limit,
and write $x \bfcdot \, u_\infty$ for $\jEut(x)$ ($t>0$, $x\in \XEut$).
Thus
\[
( x \bfcdot \, u_r ) \bfcdot \, u_\infty = x \bfcdot \, u_\infty \in \XEu
\qquad
(r,t>0, x \in \XEut).
\]
Finally, define isometries $( \SEut )_{t>0}$ on $\XEu$
by the requirement
\[
\SEut ( z \bfcdot \, u_\infty ) = u_t \bfcdot \, z \bfcdot \, u_\infty
\qquad
\big( z \in \bigcup\nolimits_{s>0} \XEus \big),
\]
and set $\SEuzero = I_{\XEu}$.
\end{notation}

As usual,
when it is expeditious to do so
we identify
$x \bfcdot \, y$ and $x \ot y = \BEst ( x \bfcdot \, y )$,
for $x \in \Es$, $y \in \Et$ and $s,t>0$.

\begin{lemma}
\label{lemma: XEu}
Let $(\E, u)$ be a pointed Arveson system.
Then
\begin{align*}
&
\XEusplust \ot u_\infty =
\XEus \ot u_\infty +
\SEus ( \XEut \ot u_\infty ), \text{ and }
\\
&
\XEu =
\XEus \ot u_\infty +
\SEus \XEu,
\qquad
(s,t>0).
\end{align*}
\end{lemma}

\begin{proof}
We drop the superscripts.
Let $s,t>0$
and set $\F = \Comp u$.
Then,
by Proposition~\ref{propn: 7.1},
\begin{align*}
X_{s+t} =
\FFminusperpsplust \ominus \Comp u_{s+t}
&\subset
( \FFminusperps \ot \FFminusperpt ) \ominus \Comp ( u_s \ot u_t )
=
X_s \ot u_t \, \op \, u_s \ot X_t \, \op \, X_s \ot X_t,
\end{align*}
but
\[
X_s \ot X_t \subset
\{ u_s \}^\perp \ot \{ u_t \}^\perp \subset
\FFminussplust \subset
X_{s+t}^\perp,
\]
so $X_{s+t} \subset X_s \ot u_t \op u_s \ot X_t$.
The reverse inclusion also holds since
\begin{align*}
X_s \ot u_t \, \op \, u_s \ot X_t
&=
( \FFminusperps \ominus \Comp u_s ) \ot u_t \, \op \, u_s \ot ( \FFminusperpt \ominus \Comp u_t )
\\
&=
( \FFminusperps \ot u_t \, \op \, u_s \ot  \FFminusperpt )\ominus \Comp u_{s+t}
\subset
\FFminusperpsplust \ominus \Comp u_{s+t}
=
X_{s+t}.
\end{align*}
The first identity follows. The second follows from the first.
\end{proof}

\begin{lemma}
\label{lemma: SEu}
Let $(\E, u)$ be a pointed Arveson system.
Then $\SEu := ( \SEut )_{t \ges 0}$
is a strongly continuous one-parameter semigroup of isometries.
Moreover it is purely isometric.
\end{lemma}

\begin{proof}
Clearly
$\SEu$
is a one-parameter semigroup of isometries.
Let $x \in \XEup$ and $y \in \XEuq$ where $p,q>0$.
Fix $T>0$ such that $T > \max \{ p, q+1 \}$.
Then, for $0 \les t \les 1$,
\begin{align*}
&
\ip{ x \ot u_\infty }{ u_t \ot y \ot u_\infty }
=
\ip{ x \ot u_{T-p} }{ u_t \ot y \ot u_{T-q-t} }
=
\ip{ x \ot u_{T-p} }{ \UETt ( y \ot u_{T-q} ) }
\end{align*}
where $\UET = ( \UETt )_{t \in \Real}$ is the unitary flip group on $\ET$
defined in Definition~\ref{defn: 1.2}.
Weak continuity of the semigroup $\SEu$ therefore follows
from the strong continuity of $\UET$.
Since weak continuity implies strong continuity for
one-parameter semigroups on Banach spaces,
the first part follows.

For the last part, let $s,t>0$.
Then
\[
u_t \ot z \ot u_\infty \perp x \ot u_s \ot u_\infty = x \ot u_\infty
\qquad
( z \in \XEus, x \in \XEut ).
\]
It follows that $\Ran \SEut \perp \Ran \jEut$ ($t>0$),
so
\[
\bigcap_{t>0} \Ran \SEut \subset
\bigcap_{t>0} ( \Ran \jEut )^\perp =
\Big( \bigcup_{t>0} \Ran \jEut \Big)^\perp = \{ 0 \},
\]
and therefore $\SEu$ is purely isometric.
\end{proof}

By Cooper's Theorem (\cite{Cooper};
see Theorem 9.3, Chapter III of~\cite{NaF-harmonic}),
it follows from Lemma~\ref{lemma: SEu} that,
for any pointed Arveson system $(\E, u)$
there is a Hilbert space $\kEu$ and unitary operator
$\VEu: \XEu \to \KEu:= L^2( \Rplus; \kEu)$
such that $\VEu \SEut = S^{\kEu}_t \VEu$ ($t \ges 0$).
Moreover $\kEu$ is separable since $\XEu$ is.

Recall our notation
$\KEut := \big\{ g \in \KEu: \esssupp g \subset [0,t] \big\}$ ($t>0$).

\begin{lemma}
\label{lemma: cluster inclusion}
Let $(\E, u)$ be a pointed Arveson system.
Set $\FEu = ( \FFminusperpt )_{t>0}$
where $\Fock = \Comp u$.
For $t>0$, define the operator
\[
\phiEut:
\FEut = \Ft \op \XEut \to F^{\kEu}_t = \Comp \op \KEut,
\quad
\lambda u_t + x \mapsto ( \lambda, J_t^* \VEu x \bfcdot \, u_\infty )
\]
where
$J_t$ denotes the inclusion map $\KEut \to \KEu$.
Then
$\phiEu = ( \phiEut )_{t>0}$ is an isomorphism of inclusion systems.
\end{lemma}

\begin{proof}
Drop the superscripts from
$\FEur$, $\KEur$, $\XEur$, $\phiEur$, $\jEur$, $\SEur$ ($r>0$)
and
$\VEu$,
and abbreviate $\kEu$ to $\kil$.

Each operator $\phi_t$ is easily seen to be unitary.
Fix $s,t>0$.
Then
\begin{align*}
\big( \betaFkst \circ \phi_{s+t} \big) ( u_{s+t} ) =
\betaFkst (1,0)
&=
(1,0) \ot (1,0)
\\
&=
\phi_{s} ( u_{s} ) \ot \phi_{t} ( u_{t} ) =
( \phi_s \ot \phi_t ) ( \betaFst  u_{s+t} ).
\end{align*}
Also,
if
$z
=
x_s \bfcdot \, u_t + u_s \bfcdot \, x_t
=
\jmath_{s+t}^* \big( \jmath_s( x_s ) + S_s \jmath_t( x_t ) \big)
$
where
$x_s \in X_s$ and $x_t \in X_t$,
then
\begin{align*}
\big( \betaFkst \circ \phi_{s+t} \big) (z)
&=
\betaFkst \big( 0, J_{s+t}^* V ( x_s \bfcdot \, u_\infty + S_s ( x_t \bfcdot \, u_\infty )) \big)
\\
&=
\betaFkst \big( 0, J_{s+t}^* ( V ( x_s \bfcdot \, u_\infty ) + S^\kil_s V ( x_t \bfcdot \, u_\infty ) ) \big)
\\
&=
( 0, J_{s}^* V ( x_s \bfcdot \, u_\infty ) ) \ot (1,0) + (1,0) \ot ( 0, J_{t}^* V ( x_t \bfcdot \, u_\infty ) )
\\
&=
 \phi_s ( x_s ) \ot \phi_t ( u_t ) +  \phi_s ( u_s ) \ot \phi_t ( x_t )
=
( \phi_s \ot \phi_t ) ( \betaFst z ).
\end{align*}
Since
$F_{s+t} = \Comp u_{s+t} \op \jmath_{s+t}^* \big( \jmath_s ( X_s ) + S_s \jmath_t ( X_t ) \big)$,
it follows that
$\betaFkst \circ \phi_{s+t} = ( \phi_s \ot \phi_t ) \circ \betaFst$.
Therefore
$\phi$ is an isomorphism of inclusion systems.
\end{proof}

\begin{theorem}
\label{thm: 5.5}
Let $\E$ be a spatial Arveson system.
Then, for any normalised unit $u$ of $\E$,
\[
( \Comp u )\check{} = \EI.
\]
\end{theorem}

\begin{proof}
Let $u \in \UnitEnorm$ and set $\F = \Comp u$.

The isomorphism of inclusion systems $\phiEu$,
defined in Lemma~\ref{lemma: cluster inclusion},
 lifts to an isomorphism of product systems
$\psi: \F\,\check{} \to \Fock^{\kEu}$.
Theorems~\ref{thm: 7.16} and~\ref{thm: always}
imply
 that $\psi$ is an isomorphism of Arveson systems.
Thus $\F\,\check{}$ is of type I, and so is contained in $\EI$.

Now let $a \in \RootEu$ and $t>0$.
Then
\begin{align*}
a_t
&=
a_r \ot u_{t-r} + u_r \ot a_{t-r}
\in
\FFperpr \ot \F_{t-r} \, \op \, \F_r \ot \FFperptminusr
\subset
\big( \FFperpr \ot \FFperptminusr \big)^\perp
\qquad
(0 < r < t),
\end{align*}
so $a_t \in \FFminusperpt$.
By Theorem~\ref{thm: 3.7},
the product subsystem of $\E$ generated by $u$ and all of its roots is $\EI$,
therefore $\F\,\check{}$ contains $\EI$.
The result follows.
\end{proof}

Before turning to its connection with the Cantor--Bendixson derivative
applied to random closed sets (in the closed unit interval),
we briefly mention
a natural generalisation of our cluster construction.
For an ordered pair of inclusion subsystems $F = (F^1, F^2)$
of an Arveson system $\E$, and $t>0$, set
\[
\Fminusperpt := \Et \ominus \Fminust
\ \text{ where } \
\Fminust :=
\vee_{0<r<t} ( \Er \ominus F^1_r ) \ot ( \Etminusr \ominus F^2_{t-r} ).
\]
this extends the earlier construction
(for a single inclusion subsystem $F$ of $\E$)
as follows:
\[
(F,F)^{\ominus \perp}_t = \Fminusperpt
\qquad
(t>0).
\]

\begin{proposition}
\label{propn: 5.6}
Let $F = (F^1, F^2)$ be an ordered pair of inclusion subsystems
of an Arveson system $\E$.
Then $\Fminusperp := \big( \Fminusperpt \big)_{t>0}$
is an inclusion subsystem of $\E$ containing $F^1$ and $F^2$.
\end{proposition}

\begin{proof}
Let $s,t>0$.
For $0 < r < t$,
\[
( F^1_r )^\perp \ot ( F^2_{t-r} )^\perp
\subset
( F^1_r )^\perp \ot \E_{t-r}
\subset
( F^1_r  \ot F^1_{t-r} )^\perp
\subset
( F^1_t )^\perp,
\]
so
$\Fminust \subset ( F^1_t )^\perp$,
thus
$F^1_t \subset \Fminusperpt$;
also
\[
\Es \ot ( F^1_r )^\perp \subset
( F^1_s  \ot F^1_{r} )^\perp \subset
( F^1_{s+r} )^\perp,
\]
so
\[
\Es \ot ( F^1_r )^\perp \ot ( F^2_{t-r} )^\perp
\subset
( F^1_{s+r} )^\perp \ot ( F^2_{t-r} )^\perp
\subset
\Fminussplust,
\]
thus
$\Es \ot \Fminust \subset \Fminussplust$.

By symmetry,
$F^2_t \subset \Fminusperpt$
and
$\Fminuss \ot \Et \subset \Fminussplust$.
Therfore
\[
\Fminusperpsplust \subset
( \Es \ot \Fminusperpt ) \cap ( \Fminusperps \ot \Et )
=
\Fminusperps \ot \Fminusperpt.
\]
It follows that
$\Fminusperp$ is an inclusion system containing $F^1$ and $F^2$.
\end{proof}

This completes the treatment of our cluster construction for product subsystems.
In order to relate it to random closed sets
we summarise the basic relevant properties of hyperspaces next.
Thus let $X$ be a topological space.
The Vietoris topology on $K(X)$, the collection of compact subsets of $X$,
has
$\{ H_U: U \text{ open in } X \}
\cup
\{ M_F: F \text{ closed in } X \}$
as sub-base
(\cite{Kecheris});
the \emph{hit sets} and \emph{miss sets} of $K(X)$ being defined as follows:
\[
H_A := \big\{ Z \in K(X): Z \cap A \neq \emptyset \big\}
\ \text{ and } \
M_A := \big\{ Z \in K(X): Z \cap A = \emptyset \big\}
\qquad
(A \subset X).
\]
Note that,
for $A,B \subset X$ and $\mathcal{A} \subset \mathcal{P}(X)$,
the following hold:
$
\{ Z \in K(X): Z \subset A \} = M_{A^{\complement}}
$,
\begin{subequations}
\label{eqn: 6.0}
\begin{align}
\label{eqn: 6.0a}
&
M_A = ( H_A )^\complement,
\
H_{\bigcup \mathcal{A}} = \bigcup\nolimits_{A \in \mathcal{A}} H_A,
\
H_\emptyset = \emptyset
\text{ and }
\{ \emptyset \} = M_X,
\
\text{ so }
\\
\label{eqn: 6.0b}
&
A \subset B \implies H_A \subset H_B,
\
M_{\bigcup \mathcal{A}} = \bigcap\nolimits_{A \in \mathcal{A}} M_A,
\
M_\emptyset = K(X)
\text{ and }
\{ \emptyset \} = ( H_X )^\complement.
\end{align}
\end{subequations}
Thus
$\emptyset$ is an isolated point of $K(X)$,
 and
a nonempty basic open set of $K(X)$ takes the form
$
B =
M_F \cap H_{U_1} \cap \cdots \cap H_{U_n}
$
for some set $F$ closed in $X$,
$n \in \Nat$ and
sets $U_1, \cdots , U_n$ open in $X$
such that
$F^\complement \cap U_i \neq \emptyset$ for $i = 1, \cdots , n$.
Note also that,
for a sequence $(F_n)$ of closed sets of $X$,
\begin{equation}
\label{eqn: FnF}
F_n \downarrow F
\implies
M_F = \bigcup_{n=1}^\infty M_{F_n}.
\end{equation}

For any dense subset $D$ of $X$,
$K_{00}(X) \cap \mathcal{P}(D)$ is dense in $K(X)$,
where
$K_{00}(X)$
denotes the collection of subsets of $X$ having finite cardinality.
If $X$ has compatible metric $d$
(with diameter at most one)
then the induced Hausdorff metric $d_{\Hausdorff}$ on $K(X)$
(for which
$d_{\Hausdorff} (Z, \emptyset) = 1 = d_{\Hausdorff}( \emptyset, Z)$
for all $Z \in K(X) \setminus \{ \emptyset \}$)
is compatible with the Vietoris topology,
and is complete if $d$ is.
If $\varepsilon>0$ and $F \subset \subset X$ is an $\varepsilon$-net
(\cite{Sutherland}, Definition 7.2.8)
with respect to a compatible metric $d$ for $X$
with diameter at most one,
then $\mathcal{P}(F)$ is an $\varepsilon$-net for $d_{\Hausdorff}$,
so
$K(X)$ is totally bounded with respect to $d_{\Hausdorff}$
if
$X$ is totally bounded with respect to $d$.
It follows from these basic facts that
$K_{00}(X)$ is dense in $K(X)$,
and $K(X)$ is
separable,
 metrisable,
completely metrisable,
Polish,
or
compact metrisable, if $X$ has that property.
When $X$ is compact Hausdorff
(so that $K(X)$ equals the collection of closed subsets of $X$),
the Vietoris topology coincides with another well-known hyperspace topology,
namely the Fell topology.

For a subset $A$ of $X$ we denote
by $A'$
 its derived set,
consisting of its points of accumulation,
$\{ x \in X: x \in \overline{ A \setminus \{ x \} } \}$.
Note that
(1)
$A' \subset \overline{A}$,
(2)
$A'$ is closed if $A$ is,
(3)
if
$X$ is a $\T_1$-space
then
$A' = \overline{A}\,'$,
in particular $A'$ is closed.
Note the further elementary properties
(assuming, for (5), that $X$ is $\T_1$):
for $A \subset B \subset X$, $C \subset X$, $U$ open in $X$ and $K \in K(X)$,
\[
(4)
\
(A \cap C )' \subset B' \cap C',
\
(5)
\
A' \cap U \neq \emptyset \implies \# ( A \cap U ) = \infty;
\
(6)
\
K' = \emptyset \iff \# K <  \infty.
\]
Thus,
for $X$ Hausdorff,
the prescription $Z \mapsto Z'$
defines a map $\Delta_X: K(X) \to K(X)$,
the \emph{Cantor--Bendixson derivative}
(whose study,
as an operator, was initiated by Kuratowski; see~\cite{Kur}).

We now turn to the connection with random closed sets.
Set $\closed := K(\mathbb{I})$,
$\closed_{00} := K_{00}(\mathbb{I})$
and
$\Delta := \Delta_{\mathbb{I}}$,
where
$\mathbb{I}$ denotes the unit interval $[0,1]$ with its standard topology.
Thus $\closed$ is compact
and metrised by the Hausdorff metric of the standard metric of $\mathbb{I}$,
in particular it is second countable,
with countable dense subset
$\closed_{00} \cap \mathcal{P}( \mathbb{I} \cap \Rat )$,
and
$\Delta^{-1}( \{ \emptyset \} ) = \closed_{00} \subsetneq \closed$.
By a \emph{random closed subset of} $\mathbb{I}$ is meant simply a
$\closed$-valued random variable,
in other words a
measurable map from $\Omega$ to $\closed$,
for a probability space $(\Omega, \mathfrak{F}, \Prob)$.

\begin{lemma}
\label{lemma: boundary}
Let $F, U \subset \mathbb{I}$, with $F$ closed, $U$ open and $F \supset U$.
Then,
the following hold.
\begin{alist}
\item
$
\Delta^{-1} ( M_F )
\subset
\big\{ Z \in \closed: \#( Z \cap F ) < \infty \big\}
\subset
\big\{ Z \in \closed: \#( Z \cap U ) < \infty \big\}
\subset
\Delta^{-1} ( M_U )
$.
\item
Let
$\partial F$ denote the topological boundary of $F$. Then
\begin{align*}
\Delta^{-1} ( M_F ) \cup H_{\partial F}
&
=
\big\{ Z \in \closed: \#( Z \cap F ) < \infty \big\} \cup H_{\partial F}
\\
&
=
\big\{ Z \in \closed: \#( Z \cap \Int F ) < \infty \big\} \cup H_{\partial F}
=
\Delta^{-1} ( M_{\Int F} ) \cup H_{\partial F}.
\end{align*}
\end{alist}
\end{lemma}

\begin{proof}
(a)
follows from (4), (6) and (5) above.

(b)
By part (a),
(b) holds with equality replaced by subset in all three places.
Let
$Z \in
\closed \setminus H_{\partial F}
=
M_{\partial F}$.
Then $Z \cap \partial F  = \emptyset$, so
$Z \cap F = Z \cap \Int F$
and so
$
Z' \cap F
=
Z' \cap \Int F
$.
Thus
$Z \in \Delta^{-1} ( M_F )$
if and only if
$Z \in \Delta^{-1} ( M_{\Int F} )$.
Therefore the outer sets coincide.
The result follows.
\end{proof}

The contents of the following proposition are known;
we include their short proofs since
they are instructive and do not seem to be readily available.

\begin{proposition}
\label{propn: X}
The following hold.
\begin{alist}
\item
$\Borel ( \closed ) =
\sigma \big\{ M_J: J \text{ is a closed subinterval of } \mathbb{I} \big\}$.
\item
$\Delta$ is Borel measurable.
\end{alist}
\end{proposition}

\begin{proof}
(a)
Denote
the RHS $\sigma$-algebra by $\Sigma$.
Let
$U$ be open in $\mathbb{I}$ and let
$F$ be closed in $\mathbb{I}$.
Then
$U = \bigcup J_n$ for
a sequence $(J_n)$ of closed subintervals of $\mathbb{I}$,
and $F = \bigcap F_n$ for
a sequence $(F_n)$ of closed sets of $\mathbb{I}$ such that
$F_n \downarrow F$ and, for each $n \in \Nat$, $F_n = \bigcup_{i=1}^{k(n)} J_n^i$
for some closed subintervals $J_n^1, \cdots , J_n^{k(n)}$ of $\mathbb{I}$.
Therefore,
using~\eqref{eqn: 6.0} and~\eqref{eqn: FnF},
\[
H_U =
\bigcup_{n=1}^\infty H_{J_n} =
\bigcup_{n=1}^\infty \big( M_{J_n} \big)^\complement \in
\Sigma
\ \text{ and } \
M_F =
\bigcup_{n=1}^\infty M_{F_n} =
\bigcup_{n=1}^\infty \bigcap_{i=1}^{k(n)} M_{J^i_n}  \in
\Sigma.
\]
Since $\closed$ is second countable
it follows from Lindel\"{o}f's Theorem that
every open set of $\closed$
is a countable union of basic open sets,
so
$\Sigma \supset \Borel( \closed )$.
The reverse inclusion is clear.

(b)
Let $J$ be a closed subinterval of $\mathbb{I}$, say $[a,b]$.
For $U$ open in $\mathbb{I}$ and $p \in \Nat$,
since $\mathbb{I}$ is Hausdorff,
the set $\{ Z \in \closed: \#(Z \cap U) \ges p \}$
equals
the open set
\[
\bigcup \big\{ H_{V_1} \cap \cdots \cap H_{V_p} :
\, V_1, \cdots , V_p \text{ disjoint open subsets of } U
\big\}.
\]
It follows from~\eqref{eqn: FnF} that
$
M_J
=
\bigcup_{n=1}^\infty M_{U_n}
=
\bigcup_{n=1}^\infty M_{\overline{U_n}}
$,
where
$U_n := ] a - \frac{1}{n}, b + \tfrac{1}{n} [ \, \cap \, \mathbb{I}$
($n \in \Nat$).
Now,
 by part (a) of Lemma~\ref{lemma: boundary},
\[
\Delta^{-1} \big( M_{ \overline{ U_n } } \big)
\subset
\big\{
 Z \in \closed: \#(Z \cap \overline{ U_n }) < \infty
 \big\}
\subset
\big\{
 Z \in \closed: \#(Z \cap U_n) < \infty
 \big\}
 \subset
 \Delta^{-1} \big( M_{ U_n }  \big)
 \]
for each $n \in \Nat$.
 It follows that
\[
 \Delta^{-1} ( M_J ) =
 \bigcup_{n \in \Nat}
 \big\{
 Z \in \closed: \#(Z \cap U_n) < \infty
 \big\}
 =
 \bigcup_{n,p \in \Nat}
 \big\{
 Z \in \closed: \#(Z \cap U_n) \ges p
 \big\}^\complement
 \in
 \Borel ( \closed ).
 \]
The Borel measurability of $\Delta$ therefore follows from part (a).
\end{proof}

\begin{remark}
$\Delta$ is not continuous,
since
$\{ \emptyset \}$ is closed in $\closed$
but
$\Delta^{-1}( \{ \emptyset \} )$
is not closed because it equals
$\closed_{00}$
which is a
dense proper subset of $\closed$.
\end{remark}

For the convenience of the reader
we quote
the key propositions upon which our next result depends.
Recall that in Liebscher's approach
the parameter set of an Arveson system $\E$ is extended to $\Rplus$,
with $\E_0 := \Comp$.

\begin{theorem}
[\cite{Lie-random}, Theorem 3.16, Proposition 3.18, Corollary 3.21]
\label{thm: Lie random}
Let $\E$ be an Arveson system,
let $P = ( P_{r,t} )_{0 \les r < t \les 1}$
be a family of nonzero orthogonal projections in the
von Neumann algebra $B( \E_1 )$
satisfying
the evolution and bi-adaptedness conditions
\begin{equation}
\label{eqn: ad-bi}
\Prs \Pst = \Prt
\ \text{ and } \
\Prt \in \IEr \ot B(\Etminusr) \ot \IEoneminust
\qquad
(0 \les r < s < t \les 1),
\end{equation}
and let $\omega$ and $\varphi$ be faithful normal states on $B( \E_1 )$.
Then the following hold\tu{:}
\begin{alist}
\item
The map
$(r,t) \mapsto \Prt$
is strongly continuous,
with $\Prt \to \IEone$
as $(r,t) \to (s,s)$
for $0<s<1$.

\item
There is a unique Borel probability measure
$\ProbPomega$ on $\closed$ satisfying
\[
\ProbPomega \Big( \bigcap\nolimits_{i=1}^N \Msiti \Big)
=
\omega \Big( \prod\nolimits_{i=1}^N \Psiti \Big)
\]
\tu{(}$N \in \Nat, 0 \les s_i < t_i \les 1 \text{ for } i=1, \cdots , N$\tu{)}.

\item
$\ProbPomega ( H_{\{ a \}} ) = 0$
\tu{(}$a \in \mathbb{I}$\tu{)}.

\item
The correspondence
$1_{\Mst} \mapsto \Pst$
$(0 \les s<t \les 1)$,
extends to an
injective normal unital representation
$
\piP: L^\infty( \ProbPomega ) \to B(\E_1)$.
Moreover,
$$
\Ran \piP = \{ \Pst: 0 \les s<t \les 1 \}''.
$$

\item
$\ProbPphi \sim \ProbPomega$.
\end{alist}
\end{theorem}

\begin{remarks}
(i)
For a product subsystem $\F$ of $\E$,
the family
$\PF = ( \PFrt )_{0 \les r<t \les 1}$,
as defined in~\eqref{eqn: PF},
satisfies~\eqref{eqn: ad-bi}.

(ii)
By (e),
the space
$L^\infty( \ProbPomega )$,
and therefore also the representation
$\piP$,
 is independent of the choice of faithful normal state $\omega$ on $B(\E_1)$.

(iii)
For a faithful normal state $\omega$ on $B( \E_1 )$,
we write $\ProbFomega$ and $\piF$ respectively
for the Borel probability measure $\ProbPomega$
and
representation $\piP$,
when $P = \PF$.
By (e),
the probability measure equivalence class of $\ProbFomega$ is independent of
the choice of faithful normal state $\omega$ on $B(\F_1)$;
let us denote it $\MF$.
\end{remarks}

We need the following extension of~\cite{Lie-random}, Corollary 6.2.

\begin{theorem}
\label{thm: MF}
Let $\F$ be a product subsystem of
an Arveson system $\E$.
Then
\[
\MF
 =
\big\{
\ProbFomega:
\omega \text{ is a faithful normal state on } B(\E_1)
\big\}.
\]
\end{theorem}
\begin{proof}
The proof in~\cite{Lie-random},
for the case where $\F$ is generated by a unit of $\E$,
works equally well for an arbitrary product subsystem.
\end{proof}

We are now ready to give our generalisation of
Proposition 3.33 of~\cite{Lie-random}.

\begin{theorem}
\label{thm: 5.2}
Let $\F$ be a product subsystem of
an Arveson system $\E$.
Then
the following hold.
\begin{alist}

\item
$
\piF \big( 1_{\Delta^{-1}(\Mst)} \big)
=
\PFcheckst
$
$(0 \les s<t \les 1)$.

\item
$
\ProbFomega \circ \Delta^{-1} = \ProbFcheckomega
$,
for any faithful normal state $\omega$ on $B( \E_1 )$.
\item
$\MF \circ \Delta^{-1} = \MFcheck$.
\end{alist}
\end{theorem}

\begin{proof}
Let $0 \les s<t \les 1$.
First note that,
by part (b) of Lemma~\ref{lemma: boundary}
and
part (c) of Theorem~\ref{thm: Lie random},
\begin{equation}
\label{eqn: fin}
\piF \big( 1_{\Delta^{-1}(\Mst)} \big)
=
\piF
\big(
1_{ \{ Z \in \closed: \# ( Z \cap [s,t] ) < \infty \} }
\big).
\end{equation}
For $Z \in \closed$,
\[
\# \big( Z \cap [s,t] \big) \ges 2
\iff
\exists_{u \in ]s,t[}:
Z \in H_{[s,u[} \cap H_{[u,t]}.
\]
and for
$0 \les a < b \les 1$,
$
\piF \big( 1_{ H_{[a,b[} } \big)
=
\piF \big( 1_{ H_{[a,b]} } \big)
$,
and
\begin{align*}
&
\piF \big( 1_{ H_{[a,b]} } \big)
=
\IEone - \PFab
=
\IEa \ot P_{ \Fock_{b-a}^\perp } \ot \IEoneminusa,
\end{align*}
so
\[
\piF \big( 1_{ H_{[s,u[} \cap H_{[u,t]} } \big)
=
\IEs \ot P_{ \Fock_{u-s}^\perp } \ot P_{ \Fock_{t-u}^\perp } \ot \IEoneminust,
\qquad
( s < u < t ).
\]
By the normality of $\piF$,
it follows that
\[
\piF
\big(
1_{ \{ Z \in \closed: \# ( Z \cap [s,t] ) \ges 2 \} }
\big)
=
\sup_{ s < u < t }
\IEs \ot P_{ \Fock_{u-s}^\perp \ot \Fock_{t-u}^\perp } \ot \IEoneminust
=
\IEs \ot P_V \ot \IEoneminust
\]
where
$
V
=  \vee_{ s < u < t }
\,
\big(
\Fock_{u-s}^\perp \ot \Fock_{t-u}^\perp
\big)
=
\Fminustminuss
$.
By the evolution property,
\[
P_V^\perp=
\wedgewithlimits_{s<u<t}
\big(
\IEone - ( \IEone - \PFsu ) ( \IEone - \PFut )
\big)
=
\wedgewithlimits_{s<u<t}
\big(
 \PFsu + \PFut - \PFst
\big).
\]
It therefore follows that
\begin{equation}
\label{eqn: st}
\piF
\big(
1_{ \{ Z \in \closed: \# ( Z \cap [s,t] ) \les 1 \} }
\big)
=
\IEs \ot P_{ \Fminusperptminuss } \ot \IEoneminust
=
P^{ \Fminusperp }_{s,t}.
\end{equation}
Now
\begin{equation}
\label{eqn: CP}
\big\{ Z \in \closed: \# ( Z \cap [s,t] ) < \infty \big\}
=
\bigcup \CP
\end{equation}
where
the union is over partitions
$\mathcal{P} = \{ s = s_0 < \cdots < s_N = t \}$
and
$
\CP :=
\bigcap_{i=1}^N
\big\{
Z \in \closed: \# ( Z \cap [s_{i-1},s_i] ) \les 1
\big\}
$.
And so,
applying~\eqref{eqn: st}
with $[s_{i-1},s_i]$ in place of $[s,t]$,
\[
\piF ( 1_{\CP} ) =
\prod_{i=1}^N P^{\Fminusperp}_{s_{i-1},s_i}
\ \text{ for } \
\mathcal{P} = \{ s = s_0 < \cdots < s_N = t \}.
\]
Therefore,
by~\eqref{eqn: CP},
the normality of $\piF$,
and
the fact that the inclusion system $\Fminusperp$
generates the product system $\F\,\check{}$,
\[
\PFcheckst
=
\piF
\big(
1_{ \{ Z \in \closed: \# ( Z \cap [s,t] < \infty \} }
\big).
\]
Combined with~\eqref{eqn: fin},
this proves (a).
Now (a) implies that
\[
\omega \Big(
\prod\nolimits_{i=1}^N \PFchecksiti
\Big)
=
( \omega \circ \piF ) \Big( 1_{ \bigcap\nolimits_{i=1}^N \Delta^{-1}(\Msiti) } \Big)
=
\big( \ProbFomega \circ \Delta^{-1} \big) \Big( \bigcap\nolimits_{i=1}^N  \Msiti \Big),
\]
for subintervals $[s_1, t_1], \cdots , [s_N, t_N]$ of $\mathbb{I}$,
so (b) follows from part (b) of Theorem~\ref{thm: Lie random}.

(c)
In view of Theorem~\ref{thm: MF},
this follows immediately from (b).
\end{proof}

%%%%%%%%%%%%%%%%%%%%%%%%%%%%%%%%%%%%%%%%%%%%% New section

 \section*{Appendix.
 Fock Arveson systems and the Guichardet picture}
% \label{B} - not referred to
 \label{section: appendix}

The symmetric Fock space over a Hilbert space $\Hil$ is denoted $\Gamma(\Hil)$.
Its exponential vectors
$\varepsilon(h) := \big( (n!)^{-1/2} h^{\ot n} \big)_{n \ges 0}$
($h \in \Hil$)
form a linearly independent and total set
which witnesses the exponential property of symmetric Fock spaces,
namely
$\Gamma(\Hil_1 \op \Hil_2) = \Gamma(\Hil_1) \ot \Gamma(\Hil_2)$
via
$\ve(h_1, h_2) \mapsto \ve(h_1) \ot \ve(h_2)$.
For any contraction $C \in B(\Hil)$,
$\Gamma(C) := \bigoplus_{n \ges 0} C^{\ot n}$ defines a contraction in
$B( \Gamma(\Hil))$
characterised by the identity
$\Gamma(C) \ve(h) = \ve(Ch)$ ($h \in \Hil$);
the map
$C \to \Gamma(C)$ is a morphism of involutive semigroups with identity,
in particular,
$\Gamma(C)$ is isometric, respectively coisometric, if $C$ is.
For $h \in \Hil$,
the \emph{Fock--Weyl operator}
is the unitary operator $W(h)$ on $\Gamma( \Hil )$
characterised by the identity
\[
W(h) \vp(k) = e^{ -i \im \ip{h}{k} } \vp( h+k ),
\ \text{ where } \
\vp(k) := e^{ - \norm{k}^2/2 } \ve(k)
\qquad
(k \in \Hil).
\]

Now let $\kil$ be a separable Hilbert space.
Set
\[
\Kil := L^2(\Rplus; \kil)
\ \text{ and } \
\Kil_t := \big\{ g \in \Kil: \esssupp g \subset [0,t] \big\}
\qquad
(t > 0),
\]
and let $S^\kil := ( S^\kil_t )_{t \ges 0}$
denote the one-parameter semigroup of unilateral shifts on $\Kil$.
The \emph{Fock Arveson system over $\kil$},
denoted $\Fockk$, is defined by
\[
\Fockkt :=
\Gamma( L^2([0,t[; \kil)) \ot \Vack_{[t,\infty[}
=
\Linbar \{ \ve(g) : g \in \Kil_t \},
\qquad(t > 0)
\]
where $\Vack_{[t,\infty[}$ denotes
the vacuum vector $\ve(0)$ in $\Gamma( L^2([t, \infty[; \kil))$,
with structure maps determined by the prescription
\[
\BFkst : \ve(h) \mapsto \ve( h_{[0,s[})  \ot \ve( (S^\kil_s)^* h ),
\ \text{ for } \
h \in \Kil_{s+t}
\qquad
(s,t > 0).
\]
It is an Arveson system
consisting of an increasing family of subspaces of the Hilbert space
$\Fockkinfty = \Gamma(\Kil)$.
Its set of normalised units is given by
$
\{
( e^{i \lambda t} \vp^c_t )_{t>0}: c \in \kil, \lambda \in \Real
\}
$
where
\[
\vp^c :=
\big( e^{- \norm{c}^2 t/2} \ve^c_t = \vp( c_{[0,t[} ) \big)_{t>0}
\ \text{ and } \
\ve^c := \big( \ve( c_{[0,t[} ) \big)_{t>0}.
\]
The \emph{vacuum unit} $\vp^0 = \ve^0$, of the Arveson system $\Fockk$, is denoted $\Vack$.

In order to describe the Guichardet picture of the Fock Arveson system over $\kil$ which
(only here in this appendix) we denote by $\Guichk$,
we need to introduce the symmetric measure space $\Gamma_t$ over the Lebesgue space $[0,t[$,
for $0 < t \les \infty$ (\cite{Guichardet}).
As a set,
\[
\Gamma_t :=
\big\{
\sigma \subset [0,t[: \# \sigma < \infty
\big\}.
 \]
Thus, denoting
 $\{
\sigma \subset [0,t[: \# \sigma = n
\}$
by $\Gamma_t^{(n)}$,
$\bigcup_{n\ges 0} \Gamma_t^{(n)}$ is a partition of $\Gamma_t$.
Since, for each $n \in \Nat$, the map
\[
\Delta^{(n)}_t := \big\{ {\bf s} \in \Rplus^n: s_1 < \cdots < s_n < t \big\}
\to
\Gamma_t^{(n)},
\quad
{\bf s} \mapsto \{ s_1, \cdots , s_n \}
\]
is bijective,
Lebesgue measure on $\Delta^{(n)}_t$ induces a measure on $\Gamma_t^{(n)}$
and thereby an isometric isomorphism
$
L^2( \Delta^{(n)}_t )
\to
L^2( \Gamma_t^{(n)} )
$,
which ampliates to an isometric isomorphism
$
L^2( \Delta^{(n)}_t ; \kil^{\ot n} )
\to
L^2( \Gamma_t^{(n)}; \kil^{\ot n} )
$.
Therefore,
composing with the isometric isomorphism
\[
L^2_{\sym} ( [0,t]^n; \kil^{\ot n} ) \to L^2( \Delta^{(n)}_t ; \kil^{\ot n} ),
\quad
F \mapsto \sqrt{n!} \, F |_{ \Delta^{(n)}_t }
\]
gives an isometric isomorphism
\[
L^2_{\sym} ( [0,t]^n; \kil^{\ot n} )
\to
L^2( \Gamma_t^{(n)}; \kil^{\ot n} )
\qquad
(n \in \Nat).
\]
By declaring that $\emptyset \in \Gamma_t^{(0)} \subset \Gamma_t$
is an atom of measure one,
we arrive at an isometric isomorphism
\[
\Fockkt
\cong
\Comp \op \bigoplus_{n=1}^\infty L^2_{\sym} ( [0,t]^n; \kil^{\ot n} )
\cong
\Comp \op \bigoplus_{n=1}^\infty L^2( \Gamma_t^{(n)}; \kil^{\ot n} )
\cong
\Guichkt,
\]
where
$\Guichkt := \big\{ G \in \Guichkinfty: \esssupp G \subset \Gamma_t \big\}$
and,
in terms of $\Phi( \kil)$, the full Fock space over $\kil$,
\[
\Guichkinfty :=
\big\{
G \in L^2( \Gamma_\infty; \Phi(\kil)):
G(\sigma) \in \kil^{ \ot \# \sigma} \text{ for a.a. } \sigma
\big\}.
\]
These isomorphisms are restrictions of a single isomorphism
$\Fockkinfty \to \Guichkinfty$, under which
$\ve(g)$ maps to $\pi_g$, for $g \in \Kil$, where
\[
\pi_g(\sigma) :=
\left\{
 \begin{array}{ll}
 1
 \in \Comp = \kil^{\ot 0}
 & \text{ if }
 \sigma = \emptyset,
 \\
 g(s_1) \ot \cdots \ot g(s_n)
 \in \kil^{\ot n}
 & \text{ if }
 \sigma = \{ s_1 < \cdots < s_n \}
 \end{array}, \right.
\]
in particular,
$\ve(0) \mapsto \delta_\emptyset$.
Moreover,
for $G \in \Guichkinfty$ and $t \ges 0$,
\[
\big( \Gamma(S^\kil_t) G \big)(\sigma) =
\left\{
 \begin{array}{ll}
 G( \sigma - t)
 & \text{ if }
 \sigma \subset [t, \infty[
 \\
 0
 & \text{ otherwise }
 \end{array}. \right.
\]
The corresponding structure maps in the Guichardet picture
are given by the prescription
\[
\BGkst H: (\alpha, \beta) \mapsto
H \big( \alpha  \cup ( \beta + s ) \big) 1_{\Gamma_{[0,s[} \times \Gamma_{[0,t[}} (\alpha, \beta)
\qquad
( H \in \Guichksplust ).
\]

For further details on Fock space and the Guichardet picture,
see~\cite{Lin-QSA} and~\cite{Par-QSC}.

\section*{Note added in proof}
Additive units for the coisometric measurable counterpart to inclusion systems
(called super product systems) have been independently introduced and applied in~\cite{MS1}; see also~\cite{MS2}.

\section*{Acknowledgements}
Valuable discussions with Volkmar Liebscher are gratefully acknowledged.
This work was supported by
the UK-India Education and Research Initiative (UKIERI),
 under the research collaboration grant
 \emph{Quantum Probability, Noncommutative Geometry \& Quantum
 Information}.
 The third-named author was also suported by
 DST-Inspire Fellowship IFA-13 MA 20.

\end{document}